\documentclass[11pt, a4paper]{article}
\usepackage[pagebackref]{hyperref}

\usepackage{xspace} 

\usepackage[english]{babel}
\usepackage{amsmath, mathrsfs}
\usepackage{amsthm}
\usepackage{amssymb}
\usepackage{graphicx}
\usepackage{textcomp}
\usepackage{enumerate}
\usepackage{bbm}
\usepackage{latexsym}

\usepackage[usenames,dvipsnames,table]{xcolor}
\usepackage[top=3cm, bottom=3cm, left=2.3cm, right=2.3cm]{geometry}

\newtheorem*{assumption}{Assumption}

\catcode`\>=\active \def>{
\fontencoding{T1}\selectfont\symbol{62}\fontencoding{\encodingdefault}}
\newcommand{\assign}{:=}
\newcommand{\cdummy}{\cdot}
\newcommand{\mathd}{\mathrm{d}}

\newcommand{\nocomma}{}

\newcommand{\tmop}[1]{\ensuremath{\operatorname{#1}}}
\newcommand{\tmstrong}[1]{\textbf{#1}}
\newcommand{\tmtextbf}[1]{{\bfseries{#1}}}
\newcommand{\tmtextit}[1]{{\itshape{#1}}}
\newenvironment{enumerateroman}{\begin{enumerate}[i.] }{\end{enumerate}}
\newtheorem{theorem}{Theorem}[section]
\newtheorem{corollary}[theorem]{Corollary}
\newtheorem{definition}[theorem]{Definition}
\newtheorem{lemma}[theorem]{Lemma}
\newtheorem{remark}[theorem]{Remark}
\newtheorem{proposition}[theorem]{Proposition}

\newcommand{\CS}{\mathscr{S}}

\newcommand{\CG}{\mathscr{G}}

\newcommand{\CH}{\mathscr{H}}

\newcommand{\CE}{\mathscr{E}}

\newcommand{\CL}{\mathcal{L}}
\newcommand{\CQ}{\mathscr{Q}}
\newcommand{\dd}{\mathrm{d}}
\newcommand{\R}{\mathbb{R}}
\newcommand{\E}{\mathbb{E}}
\newcommand{\N}{\mathbb{N}}
\newcommand{\1}{\mathbf{1}}
\newcommand{\Z}{\mathbb{Z}}

\begin{document}

\title{The Kardar--Parisi--Zhang equation\\ as scaling limit of weakly asymmetric\\ interacting Brownian motions}
\author{
  Joscha Diehl\thanks{Finanical support by the DAAD P.R.I.M.E. program is gratefully acknowledged.}\\
  Max-Planck Institute for Mathematics in the Sciences \\
  Leipzig \\
  \texttt{diehl@mis.mpg.de}
  \and
  Massimiliano Gubinelli\\
  Hausdorff Center for Mathematics\\
   \& Institute for Applied Mathematics\\
   Universit{\"a}t Bonn \\
  \texttt{gubinelli@iam.uni-bonn.de}
  \and
  Nicolas Perkowski\thanks{Financial support by the DFG via Research Unit FOR 2402 is gratefully acknowledged.} \\
  Institut f\"ur Mathematik \\
  Humboldt--Universit\"at zu Berlin \\
  \texttt{perkowsk@math.hu-berlin.de}
}

\maketitle

\begin{abstract}
   We consider a system of infinitely many interacting Brownian motions that models the height of a one-dimensional interface between two bulk phases. We prove that the large scale fluctuations of the system  are well approximated by the solution to the KPZ equation provided the microscopic interaction is weakly asymmetric. The proof is based on the martingale solutions of Gon\c calves and Jara~\cite{bib:goncalvesJara} and the corresponding uniqueness result of~\cite{bib:energyUniqueness}.
\end{abstract}


\section{Introduction}

The Kardar--Parisi--Zhang (KPZ) universality class has been at the center of a very active field of research during the past decade. The main aim is to understand and prove the \emph{strong KPZ universality conjecture}, which roughly speaking states that any model describing the random dynamics of a one-dimensional interface between two phases and satisfying a few qualitative assumptions exhibits large time fluctuations that converge under the characteristic ``1-2-3 KPZ scaling'' to a universal limit. For example, we can model the interface dynamics by independent Brownian motions that interact with their neighbors through a potential function $V$:
\begin{equation}\label{eq:general}
   \mathrm{d}\phi_t (i) = \left\{ pV' (\phi_t (i + 1) - \phi_t (i)) - qV' (\phi_t (i) - \phi_t (i-1)) \right\} \mathrm{d}t + \mathrm{d}W_t (i), \qquad (t,i) \in \R_+ \times \Z,
\end{equation}
where $(W (i) : i \in \mathbb{Z})$ is an independent family of standard Brownian motions and $p,q \ge 0$ with $p+q=1$. The case $p=q=1/2$ describes a type of balance between the two phases, so the dynamics are reversible and there is no real growth. In that case the model is well understood and it exhibits Gaussian fluctuations on the ``1-2-4 Edwards-Wilkinson (EW) scale''~\cite{bib:spohn, bib:zhu, bib:changYau}. If $p \neq q$ it is in the KPZ universality class and a general understanding of the fluctations seems to be out of reach at the moment. Indeed most of the research on the KPZ universality class is based on \emph{stochastically integrable} models which admit tractable formulas for key probabilities of interest that lend themselves to deriving explicit large time asymptotics. The main difficulty is that the universal limit, the so-called \emph{KPZ fixed point}, is not very well understood yet, although some conjectural descriptions are available~\cite{bib:corwinQuastelRemenik}.

The \emph{weak KPZ universality conjecture} on the other hand claims that the \emph{KPZ equation} is a universal object. The KPZ equation is a nonlinear stochastic PDE for the height variable $h(t,x)$ which reads
\begin{equation}
\label{eq:kpz}
   \partial_t h(t,x) = \nu \Delta h(t,x) + \lambda (|\nabla h(t,x)|^2 - \infty) + \sqrt{D} \xi(t,x), \qquad (t,x) \in \R_+ \times \R,
\end{equation}
where $\xi$ is a space-time white noise and $\nu>0, D>0, \lambda$ fixed parameters. This equation is not scale-invariant but \emph{locally subcritical} in the language of Hairer~\cite{bib:hairerRegularity}, meaning that on small time-scales the linear part of the equation dominates the nonlinear part. Due to this lack of scale-invariance the equation cannot arise as the scaling limit of one fixed model. However, the conjecture claims that if a one-dimensional stochastic interface growth model satisfies a few basic qualitative assumptions, and if it admits a tuning parameter that allows to interpolate between a reversible regime describing a balance between the two phases and a non-symmetric regime describing an imbalance, then by calibrating the model to be closer and closer to the reversible regime we can make the large time fluctuations converge to the KPZ equation. In our model~\eqref{eq:general} this corresponds to setting $p-q=\sqrt{\varepsilon}$ and taking the scaling limit along this changing family of \emph{weakly asymmetric} models.

The relation between the weak and the strong conjecture goes as follows. The KPZ equation~\eqref{eq:kpz} has to be understood as a one parameter family of models which interpolates between the EW and the KPZ fixed point in the sense of Wilson's renormalization group picture of universality in statistical physics. Indeed by scale transformations the parameters $D$ and $\nu$ in~\eqref{eq:kpz} can be set to $1$ and we remain with only one parameter $\lambda$ controlling the size of the non-linearity. When $\lambda \to 0$ we converge to a Gaussian limit while when $\lambda \to \infty$ we should be asymptotically describing the KPZ fixed point. This last however appears as a singular limit of the KPZ equation and the description of the limiting dynamics remains highly conjectural. For further information on the strong and weak KPZ universality conjectures see~\cite{bib:corwin, bib:quastel, bib:quastelSpohn, bib:spohnSurvey}.

For many years the weak KPZ universality conjecture was not much more tractable than the strong conjecture. But recent breakthroughs in the understanding of the KPZ equation have made a proof feasible for various models. The first mathematically rigorous work on the KPZ equation is due to Bertini and Giacomin~\cite{bib:bertiniGiacomin} who define the solution $h$ to the KPZ equation~\eqref{eq:kpz}
via the Cole-Hopf transform as $h\assign (\nu / \lambda) \log Z$, where $Z$ solves the linear stochastic heat equation
\[
   \partial_t Z(t,x) = \nu \Delta Z(t,x) + \frac{\lambda \sqrt{D}}{\nu} \xi(t,x), \qquad (t,x) \in \R_+ \times \R.
\]
While Bertini and Giacomin do not make sense of the equation~\eqref{eq:kpz} for $h$ it turns out that the object they define is indeed the physically relevant one and in~\cite{bib:bertiniGiacomin} they are able to show that the fluctuations of the weakly asymmetric simple exclusion process (WASEP) converge to the (derivative of the) KPZ equation. Their proof crucially relies on the fact that the WASEP behaves well under exponentiation and their approach only extends  to very specific models that allow for a useful Cole-Hopf transform on the level of the microscopic system. In recent years several other models have been identified for which this is the case, see~\cite{bib:demboTsai, bib:corwinTsai, bib:corwinShenTsai, bib:labbe}, but a general proof of universality will not be possible with the Cole-Hopf approach because it only provides us with an equation for $e^{(\lambda/\nu)h}$.

The main difficulty in interpreting the equation for $h$ is that at fixed times $t>0$ the map $x \mapsto h(t,x)$ has no better regularity than the Brownian motion (indeed the law of the Brownian motion is essentially invariant under the dynamics of $h$, see~\cite{bib:funakiQuastel}), and therefore the nonlinearity $(|\nabla h(t,x)|^2 - \infty)$ is the square of a distribution which a priori does not make any sense. This problem was finally overcome in 2011, when Hairer~\cite{bib:hairerKPZ} used a partial series expansion and rough path integration in order to construct the nonlinearity as a continuous bilinear functional on a suitable function space and then to solve the equation. Equivalent constructions have also been carried out using the theory of regularity structures, see~\cite{bib:hairerRegularity, bib:frizHairer}, or paracontrolled distributions~\cite{bib:paracontrolled, bib:kpzReloaded}. This pathwise approach is very robust and it allows to prove convergence of many models to the KPZ equation or other singular SPDEs; see~\cite{bib:hairerQuastel, bib:hairerShen, bib:kpzReloaded, bib:hoshino} for examples where the KPZ equation arises in the limit.

However, the pathwise approach is crucially based on the concept of regularity and it requires very precise quantitative estimates, which usually are not easy to come by. Moreover, it is tailored to semilinear equations (see however the very recent work~\cite{bib:ottoWeber}), and in our example~\eqref{eq:general} we are dealing with a quasilinear system which makes it difficult to apply regularity structures or paracontrolled distributions. Instead we will rely on an alternative approach that was developed in parallel to Hairer's work by Gon\c calves and Jara~\cite{bib:goncalvesJara} and refined by Gubinelli and Jara in~\cite{bib:gubinelliJara}. In these works the KPZ equation is formulated as a martingale problem, and given a stochastic model it is relatively easy to verify whether it solves the KPZ equation in the martingale sense (in particular, only very weak quantitative estimates are required). On the other hand, until recently it was not known whether the martingale description of the solution is sufficient to characterize it uniquely. This problem was solved in~\cite{bib:energyUniqueness}, where it was shown that the refined martingale solution of~\cite{bib:gubinelliJara} is unique. So now the martingale approach provides a very powerful tool for establishing the weak KPZ universality conjecture for a  wide class of models, and it has been successfully applied to many models that still seem out of reach with the pathwise approach, see~\cite{bib:goncalvesJara, bib:goncalvesJaraSethuraman, bib:goncalvesJaraSimon, bib:francoGoncalvesSimon, bib:gubinelliHQ} for examples. The main drawback is that the method is crucially based on the stationarity of the microscopic model, and even on a quite explicit knowledge of its invariant measures. Fortunately these requirements are met quite naturally in most models.

\bigskip

Let us now get to our specific model. We let $p=(1+\sqrt{\varepsilon})/2$ and $q=(1-\sqrt{\varepsilon})/2$ with $\varepsilon\in[0,1]$ in Eq.~\eqref{eq:general} with the aim of studying small perturbations of the reversible case $\varepsilon=0$.
When $\varepsilon=0$ the system is known as the (one-dimensional) \emph{Ginzburg-Landau $\nabla \phi$ interface model}, and it has been intensely studied during the past decades. The hydrodynamic limit for $\phi$ was derived in~\cite{bib:fritz,bib:guoEtAl}. The equilibrium fluctuations were studied in~\cite{bib:spohn, bib:zhu}, and the non-equilibrium fluctuations in~\cite{bib:changYau}; see also~\cite{bib:giacominOllaSpohn} for the equilibrium fluctuations in the multidimensional case where $\Z$ is replace by $\Z^d$. The large deviations were derived in~\cite{bib:donskerVaradhan} in one dimension, and  the multidimensional case was treated in~\cite{bib:funakiNishikawa}, see also~\cite{bib:benArousDeuschel, bib:deuschelEtAl} for large deviations of the stationary distributions. A nice survey on the $\nabla \phi$ model is~\cite{bib:funaki} which contains many further references.

If $\varepsilon$ does not go to zero, the model is conjectured to be in the KPZ universality class. In the recent paper~\cite{bib:sasamotoSpohn15} a special potential $V$ is considered which formally corresponds to $V(x) = \delta(x)$ and under which the dynamics become stochastically integrable, and it is shown that the rescaled fluctuations are in the KPZ universality class; see also~\cite{bib:ferrariSpohnWeiss} for the totally asymmetric case ($\varepsilon = 1$) and~\cite{bib:borodinCorwin, bib:borodinCorwinFerrari} for the totally asymmetric case with potential $V(u) = e^{-u}$. However, in all these works the long time behavior of the one-time marginal at a single site is tracked, while here we are interested in the dynamic behavior on large temporal and spatial scales.

So let now $\varepsilon \rightarrow 0$. We will not actually study the fluctuations of $\phi$ itself, but instead we focus on the height differences $u_t(i) \assign \phi_t (i) - \phi_t (i - 1)$ which solve
\begin{align*} \mathd u_t (i) & = \frac{1}{2} \left\{ \left( 1 + \sqrt{\varepsilon} \right) (V'
   (u_t (i + 1)) - V' (u_t (i))) - \left( 1 - \sqrt{\varepsilon} \right)
   (V' (u_t (i)) - V' (u_t (i - 1))) \right\} \mathd t \\
   &\quad + \mathd W_t (i + 1) - \mathd W_t (i), \qquad (t,i) \in \R_+ \times \Z.
\end{align*}
Splitting the drift into symmetric and 
antisymmetric part (see Section \ref{sec:generatorAndStationaryMeasure}), we obtain
\[ \mathd u_t (i) = \Big( \frac{1}{2} \Delta_D [V' (u_t)] (i) +
   \sqrt{\varepsilon} \nabla^{(2)}_D [V' (u_t)] (i) \Big) \mathd t + \mathd
   \nabla^{(1)}_D W_t (i), \qquad (t,i) \in \R_+ \times \Z, \]
with
\[ \Delta_D f (i) : = f (i + 1) + f (i - 1) - 2 f (i), \hspace{2em}
   \nabla^{(1)}_D f (i) : = f (i + 1) - f (i), \]
\[ \nabla^{(2)}_D f (i) : = \frac{1}{2} (f (i + 1) - f (i - 1)) . \]
As we shall see, there is a one parameter family of stationary measures, given for $\lambda \in
\mathbb{R}$ by
\[
   {\mu}_{\lambda} (\mathd u) = \prod_{j = - \infty}^{\infty} \frac{\exp(\lambda u(j) - V (u(j)))}{Z_{\lambda}} \mathd u(j) =: \prod_{j = - \infty}^{\infty} p_\lambda(u(j)) \dd u(j)
\]
where $Z_{\lambda} > 0$ is a normalization constant. This is true both in
the symmetric ($\varepsilon = 0$) and in the non-symmetric case. We write
\[
   \rho'(\lambda) := \int_\R u \, p_\lambda(u)  \mathd u
\]
for the mean of the coordinates $u(j)$ under $\mu_\lambda$. Due to the stationarity of $u$ it easily follows from the weak law of large numbers that if $u_0 \sim \mu_\lambda$, then for every test function $\eta \in \CS$, the Schwartz functions on $\R$, and for every fixed $t\ge 0$ we have
\[
   \lim_{n \to \infty} \frac{1}{n} \sum_{k \in \Z} u_t ( k) \eta(n^{-1} k) = \rho'(\lambda) \int_\R \eta(x) \mathd x,
\]
where the convergence is in probability. So on large spatial scales $u$ looks constant. The central limit theorem shows that the fluctuations around that constant limiting profile are Gaussian. Our aim is to understand the dynamic fluctuations of $u$ on large time-scales. We fix $\lambda_0 \in \mathbb{R}$ for the rest of the paper.
Choosing the non-symmetry parameter in the SDE as
\[
   \varepsilon = \frac{1}{n}
\]
and under the rescaling $u \rightarrow v^n$ with
\[
   v^n (t, x) = n^{1 / 2} (u (n^2 t, \lfloor n^{} x  - c_n t \rfloor) - \rho'(\lambda_0)), \quad (t,x) \in \R_+ \times \R,
\]
we show that $v^n$ converges to the stochastic Burgers equation (the derivative of the KPZ equation), for a
properly chosen sequence of diverging constants $c_n$. Their presence can heuristically be explained by the fact that
the non-symmetry introduces a net transport of mass into one
direction, which has to be compensated for by observing the system in a moving frame.

To obtain an intuitive understanding why we should expect the stochastic Burgers equation as the scaling limit for the fluctuations, note that the equation satisfied by the rescaled $v^n$ is
\begin{align}
  \label{eq:rescaled eq} \notag
  \mathd v^n_t (x) &= \Big\{ \frac{1}{2} n^{1 / 2}
  \Delta_n [V' (n^{- 1 / 2} v^n_t + \rho' (\lambda_0))] (x) + n^{}
  \nabla^{(2)}_n [V' (n^{- 1 / 2} v^n_t + \rho' (\lambda_0))] (x) - c_n \nabla
  v^n_t \Big\} \mathd t \\
  &\quad
  + \mathd \nabla^{(1)}_n W^n_t (\lfloor x \rfloor), \quad \qquad (t,x) \in \R_+ \times \R, 
\end{align}
where
\[ \Delta_n f (x) := n^2 (f (x + 1 / n) + f (x - 1 / n) - 2 f (x)),
   \hspace{2em} \nabla^{(1)}_n f (x) := n (f (x + 1 / n) - f (x)), \]
\[ \nabla^{(2)}_n f (x) := \frac{n}{2} (f (x + 1 / n) - f (x - 1 / n)), \]
and $W^n$ is a Brownian motion indexed by $\mathbb{Z}$ with quadratic covariation $\mathd [W^n (i), W^n (j)]_t = n \delta_{i, j} \mathd t$. Therefore, the martingale contribution in~\eqref{eq:rescaled eq} converges to the spatial derivative of a space-time white noise. Concerning the symmetric contribution, a Taylor expansion gives
\begin{align}\label{eq:formal_linear} \nonumber
  \frac{1}{2} n^{1 / 2} \Delta_n [V' (n^{- 1 / 2} v^n_t + \rho' (\lambda_0))]
  & = \frac{1}{2} n^{1 / 2} \Delta_n \{V' (\rho' (\lambda_0)) + V'' (\rho'
     (\lambda_0)) (n^{- 1 / 2} v^n_t) + O (n^{- 1})\} \\
  & = \frac{1}{2} V'' (\rho' (\lambda_0)) \Delta_n v^n_t + O (n^{- 1 / 2}),
\end{align}
so we might guess that in the limit only the Laplace operator survives. This is true, however the diffusion constant will not
be $V'' (\rho' (\lambda_0))$. The problem with this formal
expansion is that $(n^{- 1 / 2} v^n_t)$ converges to $0$ in $\CS'$, but not in
any function space, and therefore the corrector terms appearing in the Taylor expansion cannot be controlled uniformly in $n$. Similarly we expand the asymmetric contribution on the right hand
side of~(\ref{eq:rescaled eq}) and obtain formally
\begin{align}\label{eq:formal_nonlinear} \nonumber
 &n \nabla^{(2)}_n [V' (n^{- 1 / 2} v^n_t + \rho' (\lambda_0))] - c_n
 \nabla v^n \\ \nonumber
 &\quad= n \nabla^{(2)}_n \Big\{ V' (\rho' (\lambda_0)) + V'' (\rho' (\lambda_0))
 n^{- 1 / 2} v^n_t + \frac{1}{2} V^{(3)} (\rho' (\lambda_0)) (n^{- 1 / 2}
 v^n_t)^2 + O (n^{- 3 / 2}) \Big\} - c_n \nabla v^n_t \\
 &\quad= n^{1 / 2} V'' (\rho' (\lambda_0)) \nabla^{(2)}_n v^n_t - c_n \nabla v^n_t
 + \frac{1}{2} V^{(3)} (\rho' (\lambda_0)) \nabla^{(2)}_n (v^n_t)^2 + O (n^{- 1 / 2}).
\end{align}
We see that the constants $(c_n)$ are needed to absorb the diverging linear transport term on the
right hand side and thus leave the quadratic term in the limit. The real picture is again more complicated and from this
formal expansion we did not get the right constants, but at least it gives us
an idea why in the limit we should be able to replace the complicated
nonlinear term with a quadratic contribution.

To prove the convergence we will not actually work with the piecewise constant extension of the
lattice function $u$ to a function on $\mathbb{R}$, but instead look at the unknown as a distribution in continuous space obtained by considering linear combinations of Dirac delta function at each lattice point. This  makes the notation
slightly more convenient. We are thus interested in the limit for $n
\rightarrow \infty$ of the generalized random field
\begin{equation} \label{eq:vn working def} v^n_t = \sum_k n^{1 / 2} (u_{n^2 t} (k) - \rho'
   (\lambda)) n^{- 1} \delta_{n^{- 1} k + c_n t} . \end{equation}
The main result of this paper is  the following universality result, which states that the limiting system solves the stochastic Burgers equation and only depends on the second and third centered moments of the coordinates $u(j)$ under $\mu_{\lambda_0}$, but not on the detailed shape of the potential $V$. To state it precisely, let us introduce the notations
\[
   m_{k,\lambda} := \int_\R (u - \rho'(\lambda))^k p_\lambda(u) \mathd u, \qquad \sigma^2_\lambda := m_{2,\lambda}
\]
for $k \in \N$ and $\lambda \in \R$. We also need the following assumption on $V$.

\begin{assumption}[V]\label{ass:v}
  Assume $V = \varphi + \psi$, where $\varphi \in C^2$ and there exists $C>0$ with $\varphi''(x) \in [1/C,C]$ for all $x \in \R$, and where $\psi \in C^2_b$.
\end{assumption}
\newcommand{\assumptionV}{\hyperref[ass:v]{Assumption (V)}\xspace}

\begin{theorem}\label{thm:main result}
  Let $V$ satisfy \assumptionV and
  let the rescaled stationary fluctuations of $v^n$ be given by~(\ref{eq:vn working def}), where we set
  \[ c_n \assign n^{1 / 2} \sigma_{\lambda_0}^{-2} . \]
  Then $(v^n)_{n \in \mathbb{N}}$ converges weakly in $C \left( [0, 1],
  \CS' \right)$ to the unique stationary energy solution $u$ of the
  stochastic Burgers equation
  \begin{equation}\label{eq:SBE}
     \partial_t u = \frac{ 1}{2 \sigma_{\lambda_0}^{2}}
     \Delta u - \frac{m_{3,\lambda_0}}{2 \sigma_{\lambda_0}^6} \nabla u^2 + \nabla \xi,
  \end{equation}
  where $\xi$ is a space-time white noise and the equation is interpreted in
  the sense of \cite{bib:energyUniqueness}. For all fixed times $t\ge 0$,
  $u (t)$ is a space white noise with variance $\sigma_{\lambda_0}^2$.
\end{theorem}

The proof of this theorem will occupy the rest of the paper. We will follow the general approach developed by Gon\c calves and Jara in~\cite{bib:goncalvesJara}. Of course, the convergence also holds in $C([0,T], \CS')$ for arbitrary $T>0$ or more generally locally in time on $C([0,\infty), \CS')$.

\begin{remark}
   We need the assumptions on the potential for two reasons: first, they guarantee that $V'$ is Lipschitz continuous and thus allow us to construct the dynamics of $v$ by solving the corresponding SDE. Second, we apply several results of Caputo~\cite{bib:caputo} who needs $V$ to be a perturbation of a uniformly convex function (but allows faster-than-linear growth). In principle it should be possible to relax the Lipschitz continuity assumption and to obtain the same result for all potentials that satisfy the assumptions of Caputo, or more generally of Menz and Otto~\cite{bib:menzOtto}.
\end{remark}

\begin{remark}
   At first sight the constants in front of the Laplacian and the nonlinearity in~\eqref{eq:SBE} look very different from the ones that we derived in~{\normalfont (\ref{eq:formal_linear}), (\ref{eq:formal_nonlinear})}. However, Lemma~\ref{lem:constants} below will show that
   \[
      \frac{1}{\sigma_{\lambda_0}^{2}} = \partial_{\rho} \varphi_{V'}(\rho'(\lambda_0)), \qquad -\frac{m_{3,\lambda_0}}{\sigma_{\lambda_0}^6} = \partial_{\rho\rho} \varphi_{V'}(\rho'(\lambda_0)),
   \]
   where $\varphi_{V'}(\rho) = \int_\R V'(u) p_{h'(\rho)}(u) \mathd u$ and $h'$ is the inverse function of $\rho'$, that is $p_{h'(\rho)}$ has mean $\rho$. So the correct diffusion constant $\partial_{\rho} \varphi_{V'}(\rho'(\lambda_0))$ can be interpreted as an averaged version of the constant $V''(\rho'(\lambda_0))$ appearing in~\eqref{eq:formal_linear}.
\end{remark}

The structure of the paper is as follows. Below we introduce some basic notation that we will fix throughout the paper. In Section~\ref{sec:generatorAndStationaryMeasure} we recall / rigorously show some basic features of the dynamics that are in principle well known. Section~\ref{sec:second-order-BG} is devoted to the proof of the second order Boltzmann-Gibbs principle, which is the main tool for establishing the convergence.  Section~\ref{sec:convergence} contains the convergence proof. Appendix~\ref{app:periodic approximation} collects some results on infinite-dimensional SDEs with additive noise and in Appendix~\ref{app:equiv-ensem-proof} we give the proof of a second order equivalence of ensembles result which is needed to derive the second order Boltzmann-Gibbs principle.

\paragraph{Acknowledgements} We would like to thank Tadahisa Funaki and Herbert Spohn for suggesting the problem studied in this paper.

\paragraph{Notation} If $\mu$ is a probability measure on $\R^\Z$ we will always write $\E_\mu$ for the expectation with respect to the Markov process $u$ started in the initial distribution $\mu$. That is, $\E_\mu$ is an expectation on the path space $C([0,\infty), \R^\Z)$. For $\lambda \in \R$ we write $E_\lambda$ for the expectation under $\mu_\lambda$ on $\R^\Z$ and $\langle \cdummy, \cdummy \rangle_{\lambda}$ for
the inner product in $L^2 ({\mu}_{\lambda})$, as well as $\mathrm{var}_\lambda$ for the variance under $\mu_\lambda$. We also define
\[
   \rho(\lambda):= \log Z_\lambda = \log \int_\R \exp (\lambda u - V(u)) \dd u,
\]
which is easily seen to be consistent with the previously introduced notation
\[
   \rho'(\lambda) = \int_\R u \, p_\lambda(u)  \mathd u.
\]
Furthermore, we have $\rho''(\lambda) = \sigma^2_\lambda > 0$ and therefore $\rho$ is
strictly convex. Let $h$ be its Legendre transform.
Then $\rho'( h'(x) ) = x$. That is, $h'(\rho)$ tells us which parameter $\lambda$ to choose so that $p_\lambda$ has mean $\rho$.

\section{Derivation of the generator and the stationary measure}\label{sec:generatorAndStationaryMeasure}

Throughout this section we fix $N \in \mathbb{N}$ and consider the periodic
model $v : [0, T] \times \mathbb{Z}_N \rightarrow \mathbb{R}$ (where
$\mathbb{Z}_N =\mathbb{Z}/ (N\mathbb{Z})$). We derive its stationary
measures and the decomposition of its generator in the symmetric and the
antisymmetric part, which in particular allows us to deduce the structure of
the time reversed process. The results of Appendix~\ref{app:periodic
approximation} then allow us to extend this to the model in infinite volume.
Let
\begin{align}
  \label{eq:periodic GL}\nonumber
  \mathd v_t (i) & = \frac{1}{2} \left\{ (1 + \alpha) (V'
  (v_t (i + 1)) - V' (v_t (i))) -  (1 - \alpha) (V' (v_t (i)) - V'
  (v_t (i - 1))) \right\} \mathd t \\ 
  &\qquad  + \mathd W_t (i + 1) - \mathd W_t (i),  \qquad (t,i) \in \R_+ \times \Z_N,
\end{align}
 where $(W (i) : i \in \mathbb{Z}_N)$ is a family of
independent standard Brownian motions and $\alpha \in \mathbb{R}$. We will
consider the measures
\[ {\mu}_{\lambda}^N (\mathd v) = \prod_{i \in \mathbb{Z}_N} p_{\lambda}
   (v(i)) \mathd v(i) \]
on $\mathbb{Z}_N$ and write $E_{{\mu}^N_{\lambda}}$ for the expectation
under ${\mu}^N_{\lambda}$, as well as $\langle F, G
\rangle_{{\mu}^N_{\lambda}} \assign E_{{\mu}^N_{\lambda}} [F G]$. Let $F \colon
\mathbb{R}^{\mathbb{Z}_N} \rightarrow \mathbb{R}$ be a $C^2$ function.
Applying It{\^o}'s formula to $F (v_t)$, we get
\begin{align*}
   F (v_t) - F (v_0) & = \int_0^t \sum_{i \in \mathbb{Z}_N} \partial_i F (v_s)\mathd v_s (i) + \frac{1}{2} \int_0^t \sum_{i, j \in \mathbb{Z}_N} \partial_{i, j}^2 F (v_s) \mathd \langle v (i), v (j) \rangle_s \\
   & = \int_0^t \sum_{i \in \mathbb{Z}_N} \partial_i F (v_s) \left\{ \left(\frac{1}{2} \Delta_D [V' (v_s)] (i) + \alpha \nabla^{(2)}_D [V' (v_s)] (i)
   \right) \mathd s + \mathd \nabla^{(1)}_N w_s (i) \right\} \\
   & \qquad + \frac{1}{2} \int_0^t \sum_{i \in \mathbb{Z}_N} (2 \partial_{i, i}^2 F (v_s) - \partial_{i, i - 1}^2 F (v_s) - \partial^2_{i, i + 1} F (v_s)) \mathd s.
\end{align*}
So on sufficiently nice  functions, the generator $\CG^{\alpha}$ of $v$ acts as
\begin{align*}
   \CG^{\alpha} & = \frac{1}{2} \sum_{i \in \mathbb{Z}_N} \{ - (V' (v (i + 1)) - V' (v (i))) (\partial_{i + 1} - \partial_i) + \alpha (V' (v (i + 1)) - V'(v (i - 1))) \partial_i \} \\
   &\quad + \frac{1}{2} \sum_{i \in \mathbb{Z}_N} (\partial_{i + 1} - \partial_i)^2.
\end{align*}
\begin{lemma}
  \label{lem:generator}Define the operators
  \[ \CG_S = \frac{1}{2} \sum_{i \in \mathbb{Z}_N} \{ - (V' (v (i + 1)) - V'
     (v (i))) (\partial_{i + 1} - \partial_i) + (\partial_{i + 1} -
     \partial_i)^2 \}, \]
  and
  \[ \CG_A = \frac{1}{2} \sum_{i \in \mathbb{Z}_N} (V' (v (i + 1)) - V' (v (i
     - 1))) \partial_i . \]
  Then we have $\CG^{\alpha} = \CG_S + \alpha \CG_A$ and for all $F, G \in C^2
  (\mathbb{R}^{\mathbb{Z}_N}, \mathbb{R})$ with polynomial growth of their
  partial derivatives up to order $2$ and for all $\lambda \in \mathbb{R}$
  \[ \left\langle F, \CG_S G \right\rangle_{{\mu}^N_{\lambda}} =
     \left\langle \CG_S F, G \right\rangle_{{\mu}^N_{\lambda}},
     \hspace{2em} \left\langle F, \CG_A G
     \right\rangle_{{\mu}^N_{\lambda}} = - \left\langle \CG_A F, G
     \right\rangle_{{\mu}^N_{\lambda}} . \]
  In particular, $E_{{\mu}^N_{\lambda}} \left[ \CG^{\alpha} F \right] =
  0$.
\end{lemma}

\begin{proof}
  The adjoint of
  $\partial_i$ in $L^2 ({\mu}_{\lambda})$ is given by $\partial_i^{\ast} =
  - \partial_i - (\lambda - V' (v(i)))$, and therefore
  \[ (\partial_{i + 1} - \partial_i)^{\ast} = - \{(\partial_{i + 1} -
     \partial_i) - (V' (v(i + 1)) - V' (v(i)))\}, \]
  and
  \begin{align*}
     &\{ - (V' (v (i + 1)) - V' (v (i))) (\partial_{i + 1} - \partial_i) + (\partial_{i + 1} -
     \partial_i)^2 \}^{\ast} & \\
     &\qquad = - (\partial_{i + 1} - \partial_i)^\ast \{V' (v (i + 1)) - V' (v (i))\} \\
     &\qquad \quad + (\partial_{i + 1} - \partial_i)^\ast \{- ((\partial_{i + 1} -  \partial_i) - (V' (v(i + 1)) - V' (v(i))))\} \\
     &\qquad= - (\partial_{i + 1} - \partial_i)^\ast (\partial_{i + 1} -  \partial_i) \\
     &\qquad = (\partial_{i + 1} -  \partial_i)^2 - (V' (v (i + 1)) - V' (v (i))) (\partial_{i + 1} - \partial_i),
  \end{align*}  
%
%
%
%
  from where we readily see that $(\CG_S)^\ast = \CG_S$.
  For $\CG_A$ we obtain
  \begin{align*}
   ((V' (v(i + 1)) - V' (v(i - 1))) \partial_i)^{\ast}
   &= \partial_i^{\ast} (V' (v(i + 1)) - V' (v(i - 1))) \\
   & = (- \partial_i - (\lambda - V' (v_i))) (V' (v(i + 1)) - V' (v(i - 1))) \\
   &= - (V' (v(i + 1)) - V' (v(i - 1))) \partial_i \\
   &\quad - (\lambda - V' (v(i))) (V' (v(i + 1)) - V' (v(i - 1))).
  \end{align*}
  But since it is a telescopic sum, we have
  \[ \sum_{i \in \mathbb{Z}_N} \{ - (\lambda - V' (v(i))) (V' (v(i + 1)) - V' (v(i - 1))) \} = 0, \]
  and therefore $\left( \CG_A \right)^{\ast} = - \CG_A$.
\end{proof}

\begin{corollary}
  \label{cor:reversedEvolution}For all $\alpha, \lambda \in \mathbb{R}$ the
  measure ${\mu}_{\lambda}$ is invariant under the evolution of
  \begin{equation}\label{eq:sbe with alpha}
     \mathd u_t (i) = \Big( \frac{1}{2} \Delta_D [V'
    (u_t)] (i) + \alpha \nabla^{(2)}_D [V' (u_t)] (i) \Big) \mathd t +
    \mathd \nabla^{(1)}_D W_t (i), \hspace{2em} (t,i) \in \R_+ \times \mathbb{Z}.
  \end{equation}
  Moreover, if $u_0$ is distributed according to ${\mu}_{\lambda}$ and we
  fix $T > 0$ and set $\hat{u}_t = u_{T - t}$, $t \in [0, T]$, then $\hat{u}$
  is a weak solution of the SDE
  \[ \mathd \hat{u}_t (i) = \Big( \frac{1}{2} \Delta_D [V' (\hat{u}_t)] (i) -
     \alpha \nabla^{(2)}_D [V' (\hat{u}_t)] (i) \Big) \mathd t + \mathd
     \nabla^{(1)}_D \hat{W}_t (i), \hspace{2em} (t,i) \in [0,T] \times \mathbb{Z}, \]
  where $(\hat{W} (i))_{i \in \mathbb{Z}}$ is a family of independent
  standard Brownian motions in the filtration generated by $\hat{u}$.
\end{corollary}

\begin{proof}
  Let us restrict $W$ and the initial condition $u_0$ to $[- N / 2, N / 2) \cap
  \mathbb{Z}$ and write $W^N$ and $v^N_0$ for the periodic extension of the
  restriction. Then we can interpret $W^N$ and $v^N_0$ as functions defined on
  $\mathbb{Z}_N$, and consider the solution $v^N$ to~(\ref{eq:periodic GL}) \
  started in $v_0^N$ and forced by $W^N$. For $v^N$ we can now apply
  Lemma~\ref{lem:generator} and Echeverria's criterion~\cite[Theorem 4.9.17]{bib:ethierKurtz}
  to obtain that
  ${\mu}^N_{\lambda}$ is a stationary distribution. It is then a standard
  result that the stationary time reversed process $\hat{v}^N$ is a Markov
  process with infinitesimal generator $\CG_S - \alpha \CG_A$. But since we
  are dealing with a finite dimensional diffusion, this simply means that
  $\hat{v}^N$ is a weak solution of the SDE
  \[ \mathd \hat{v}^N_t (i) = \Big( \frac{1}{2} \Delta_D [V' (\hat{v}^N_t)]
     (i) - \alpha \nabla^{(2)}_D [V' (\hat{v}^N_t)] (i) \Big) \mathd t +
     \mathd \nabla^{(1)}_D \hat{W}^N_t (i), \hspace{2em} i \in \mathbb{Z}_N .
  \]
  Alternatively, the finite dimensional case is also treated in
  \cite[Theorem 2.3]{bib:milletEtAl}.
  
  The result for $u$ and $\hat{u}$ now follows by sending $N \rightarrow \infty$, see
  Theorem~\ref{thm:periodic approximation}, because if we interpret $v^N$ as a
  periodic function on $\mathbb{Z}$, then it agrees with the solution started
  in $v_0^N$ and forced by $W^N$, and it is not hard to see that $(v_0^N, W^N)
  \rightarrow (v_0, W)$ in the spaces considered in Appendix~\ref{app:periodic
  approximation}.
\end{proof}

Let us introduce some suitable test functions to describe the generator of the dynamics in infinite volume. We say that $F \colon \mathbb{R}^{\mathbb{Z}}
\rightarrow \mathbb{R}$ is a \tmtextit{local function} if it only depends on
a finite number of coordinates, and we write $F \in C^{\infty}_c
(\mathbb{R}^{\mathbb{Z}}, \mathbb{R})$ if $F$ is a local function that has
continuous partial derivatives of every order and that is compactly supported
if we consider it as a map defined on the finite-dimensional subspace of
variables in which $F$ is not constant. Then $C^{\infty}_c
(\mathbb{R}^{\mathbb{Z}}, \mathbb{R})$ is dense in $L^2
({\mu}_{\lambda})$ for any $\lambda \in \mathbb{R}$. Indeed, local
functions can be identified with functions on $\mathbb{R}^d$ for some $d \in
\mathbb{N}$, and of course $C^{\infty}_c (\mathbb{R}^d,
\mathbb{R})$ is dense in $L^2 (\mathbb{R}^d, {\mu}_{\lambda}^d)$.
Furthermore, it follows from the martingale convergence theorem that every
function in $L^2 ({\mu}_{\lambda})$ can be approximated by local
functions; simply consider the conditional expectations on $\sigma (u (j) : |
j | \leqslant m)$ for $m \in \mathbb{N}$.

In the following we will always take $\alpha = \sqrt{\varepsilon}$ for
$\varepsilon = 1 / n$ in~(\ref{eq:sbe with alpha}). By
Corollary~\ref{cor:reversedEvolution}, the maps
\[ T_t F (u) =\mathbb{E}_u [F (u_t)], \hspace{2em} t \geqslant 0, \]
define a strongly continuous contraction semigroup on $L^2
({\mu}_{\lambda})$ (to see the strong continuity first approximate a
general $F \in L^2 ({\mu}_{\lambda})$ by functions in $C^{\infty}_c
(\mathbb{R}^{\mathbb{Z}}, \mathbb{R})$). We will denote the infinitesimal
generator of $(T_t)$ with $\CL$, while $\CL_S$, $\CL_A$, and $\CL^{\ast}$
denote its symmetric/antisymmetric part and adjoint, respectively. It follows
from Lemma~\ref{lem:generator} that for $F \in C^{\infty}_c
(\mathbb{R}^{\mathbb{Z}}, \mathbb{R})$ we have
 \begin{align*}
 \CL_S F &= \frac{1}{2} \sum_{i \in \mathbb{Z}} \{ - (V' (u (i + 1)) - V' (u
 (i))) (\partial_{i + 1} - \partial_i) + (\partial_{i + 1} - \partial_i)^2
 \} F, \\
 \CL_A F &= \frac{\sqrt{\varepsilon}}{2} \sum_{i \in \mathbb{Z}} (V' (u (i +
 1)) - V' (u (i - 1))) \partial_i F.
 \end{align*}
The generator $n^2 \CL$ of $u^n : = u_{n^2 \cdummy}$ will be denoted by
$\CL^{(n)}$, and we also write $\CL^{(n)}_S = n^2 \CL_S$ and similarly for
$\CL^{(n)}_A$. Recall that $\mathcal{L}$ also depends on the non-symmetry
parameter $\varepsilon$.

\begin{lemma}
  \label{lem:core}The space $C^{\infty}_c (\mathbb{R}^{\mathbb{Z}},
  \mathbb{R})$ is a core for $\CL$ and also for $\CL^{\ast}$, that is the
  closure of the operator $\CL |_{C^{\infty}_c (\mathbb{R}^{\mathbb{Z}},
  \mathbb{R})}$ is equal to $\CL$ and similarly for $\CL^{\ast}$.
\end{lemma}

\begin{proof}
  We concentrate on $\CL$, the
  arguments for $\CL^\ast$ are the same. 
  According to \cite[Proposition~1.3.1]{bib:ethierKurtz} it
  suffices to show that $C^{\infty}_c (\mathbb{R}^{\mathbb{Z}},
  \mathbb{R})$ is dense in $L^2 ({\mu}_{\lambda})$ (which we already
  know) and that there is some $\gamma > 0$ such that the range $R (
  \gamma - \CL |_{C^{\infty}_c (\mathbb{R}^{\mathbb{Z}}, \mathbb{R})}
  )$ is dense in $L^2
  ({\mu}_{\lambda})$.  It suffices
  to show that $R ( \gamma - \CL |_{C^{\infty}_c
  (\mathbb{R}^{\mathbb{Z}}, \mathbb{R})} )$ is dense in $C^{\infty}_c
  (\mathbb{R}^{\mathbb{Z}}, \mathbb{R}) \subset L^2 ({\mu}_{\lambda})$.
  But if $F \in C^{\infty}_c (\mathbb{R}^{\mathbb{Z}}, \mathbb{R})$ depends
  on finitely many coordinates, say $\mathbb{R}^d$, then we can find a
  solution $G : \mathbb{R}^d \rightarrow \mathbb{R}$ to the
  finite-dimensional resolvent equation $\left( \gamma - \CL \right) G = F$
  (slightly abusing notation by also writing $\CL$ for the diffusion operator
  acting on $\mathbb{R}^d$). Moreover, since the diffusion matrix of $\CL$ is
  uniformly elliptic, we get that $G \in C^{\infty}_b (\mathbb{R}^d)$
  although not necessarily $G \in C^{\infty}_c (\mathbb{R}^d)$. But if now
  $\varphi : \mathbb{R}^d \rightarrow \mathbb{R}$ is a smooth compactly
  supported function with $\varphi(0) = 1$ and $\varphi_n (x) = \varphi (x / n)$, then
  \[ \left( \gamma - \CL \right) (\varphi_n G) = \varphi_n F - \left( \CL
     \varphi_n \right) G - A \varphi_n A G, \]
  where $A$ is a first order differential operator, and since
  \[ \| \partial_i \varphi_n \|_{\infty} = \| n^{- 1} (\partial_i \varphi)
     (n^{- 1} \cdummy) \|_{\infty} \lesssim n^{- 1}, \]
  we get that $\left( \gamma - \CL \right) (\varphi_n G)$ converges to $F$ uniformly.
\end{proof}

\section{Second order Boltzmann-Gibbs principle}\label{sec:second-order-BG}

\subsection{Kipnis-Varadhan lemma and spectral gap}\label{sec:kipnisVaradhan}

Recall that we work throughout with the stationary measure
${\mu}_{\lambda_0}$ for a fixed $\lambda_0 \in \mathbb{R}$. Let us define
the space $\CH^1$ as the completion of $C^{\infty}_c
(\mathbb{R}^{\mathbb{Z}}, \mathbb{R})$ with respect to $\lVert \cdot \rVert_1$,
where
\[ \| F \|_1^2 \assign \left\langle F, - \CL_S F \right\rangle_{\lambda_0} .
\]
Since we showed in Lemma~\ref{lem:core} that $C^{\infty}_c
(\mathbb{R}^{\mathbb{Z}}, \mathbb{R})$ is a core for $\CL$ and $\CL^{\ast}$
we are in the setting of
\cite[Section~2.2]{bib:komorowskiLandimOlla}. We
also write
\[
  \CE (F) \assign  \frac{1}{2} \sum_{i \in \mathbb{Z}} (\partial_i F -
  \partial_{i + 1} F)^2 .
\]
From a similar computation as in Lemma~\ref{lem:generator} it easily follows
that $\| F \|_1^2 = E_{\lambda_0} \left[ \CE (F) \right]$ whenever $F \in
C^{\infty}_c (\mathbb{R}^{\mathbb{Z}}, \mathbb{R})$. We then define for $F \in L^2 ({\mu}_{\lambda_0})$
\[ \| F \|_{- 1}^2 \assign \sup_{G \in C^{\infty}_c
   (\mathbb{R}^{\mathbb{Z}}, \mathbb{R})} \{ 2 \langle F, G
   \rangle_{\lambda_0} - \| G \|_1^2 \}, \]
and the space $\CH^{- 1}$ is defined as the completion of $\{ F \in L^2
({\mu}_{\lambda_0}) : \| F \|_{- 1}^2 < \infty \}$, after identifying all
functions $F$ and $G$ that satisfy $\| F - G \|_{- 1} = 0$. Note that $\| F
\|_{- 1} = \infty$, for all $F \in L^2 ({\mu}_{\lambda_0})$ with
$E_{\lambda_0} [F] \neq 0$, since $\| c \|_1 = 0$ for all constants $c$, so in
particular every $F \in L^2 ({\mu}_{\lambda_0}) \cap \CH^{- 1}$ satisfies
$E_{\lambda_0} [F] = 0$.

The following result can be shown in the same way as \cite[Lemma~2]{bib:gubinelliJara}.

\begin{lemma}[It{\^o} trick]
  For any $T > 0$ and $H \in L^2([0,T]; \CH^{1})$ we
  have
  \[
     \mathbb{E}_{{\mu}_{\lambda_0}} \Big[ \sup_{t \in [0, T]} \Big(\int_0^t \CL^{(n)}_S H (s, u^n_s) \mathd s \Big)^2 \Big] \lesssim n^2 \int_0^T \| H (s, \cdummy) \|_1^2 \mathd s.
  \]
\end{lemma}

Using the It{\^o} trick, it is not hard to obtain the Kipnis-Varadhan lemma in
our context. The proof is the same as for \cite[Lemma~2]{bib:energyUniqueness}.

\begin{lemma}[Kipnis-Varadhan lemma]
  \label{lem:Kipnis-Varadhan}
  For any $T > 0$ and $F \in L^2( [0, T]; \CH^{- 1} \cap L^2)$ we have
  \[ \mathbb{E}_{{\mu}_{\lambda_0}} \Big[ \sup_{t \in [0, T]} \Big(
     \int_0^t F (s, u^n_s) \mathd s \Big)^2 \Big] \lesssim n^{- 2}
     \int_0^T \| F (s, \cdummy) \|_{- 1}^2 \mathd s. \]
\end{lemma}

Note that if $F (s, \cdummy) = \CL_S^{(n)} H (s, \cdummy)$ for all $s \in [0,
T]$, then we do not lose anything by applying the Kipnis-Varadhan lemma rather
than the It{\^o} trick: Since $\CL^{(n)}_S = n^2 \CL_S$, we get from
Lemma~\ref{lem:Kipnis-Varadhan}
\[ \mathbb{E}_{{\mu}_{\lambda_0}} \Big[ \sup_{t \in [0, T]} \Big(
   \int_0^t \CL^{(n)}_S H (s, u^n_s) \mathd s \Big)^2 \Big] \lesssim n^2
   \int_0^T \| \CL_S H (s, \cdummy) \|_{- 1}^2 \mathd s. \]
But $\CL_S$ is an isometry between $\CH^1$ and $\CH^{- 1}$ (see \cite[Claim~D on
p.44]{bib:komorowskiLandimOlla}), so that
\[ \| \CL_S H (s, \cdummy) \|_{- 1}^2 = \| H (s, \cdummy) \|_1^2, \]
which leads to the same bound that we get from the It{\^o} trick, at least up
to multiplication with a constant. However, we should point out that the It\^o trick allows in principle to control higher order moments, while the Kipnis-Varadhan lemma is limited to $L^2$ estimates.

To get a useful bound for the $\lVert \cdummy \rVert_{- 1}$ norm, we will use a spectral gap estimate. But first let us
introduce more notation: For $\ell \in \mathbb{N}$ and $\nu \in \mathbb{R}$
we define the measure $\nu_{\rho}^{\ell} ={\mu}^{\ell}_{\rho} (\cdummy |
\sum_{i = 0}^{\ell - 1} u(i) = \ell \rho)$. Note that $\nu^{\ell}_{\rho}$ does
not depend on $\lambda$. Indeed, under $\nu^{\ell}_{\rho}$ the density of
$(u(0), \ldots, u(\ell - 2))$ is given by
\[ \frac{p_{\lambda} (u(0)) \ldots p_{\lambda} (u(\ell - 2)) p_{\lambda}
   ( \ell \rho - \sum_{i = 0}^{\ell - 2} u(i) )}{p_{\lambda}^{\ast
   \ell} (\ell \rho)}, \]
and
\[
   p_{\lambda}^{\ast \ell} (\ell \rho) = \int_{\mathbb{R}^{\ell - 1}} \mathd u(1) \ldots \mathd u(\ell - 1) p_{\lambda} (\ell \rho - u(0) - \ldots - u(\ell - 2)) p_{\lambda} (u(0)) \ldots p_{\lambda} (u(\ell - 2)) = Z^{- \ell}_{\lambda} e^{\lambda \ell \rho} p_0^{\ast \ell} (\ell \nu),
\]
and therefore
\begin{equation}\label{eq:conditional density}
   \frac{p_{\lambda} (u(0)) \ldots p_{\lambda}
  (u(\ell - 2)) p_{\lambda} ( \ell \rho - \sum_{i = 0}^{\ell - 2} u(i)
  )}{p_{\lambda}^{\ast \ell} (\ell \rho)} = \frac{p_0 (u(0)) \ldots p_0
  (u(\ell - 2)) p_0 ( \ell \rho - \sum_{i = 0}^{\ell - 2} u(i)
  )}{p_0^{\ast \ell} (\ell \rho)}
\end{equation}
for all $\lambda \in \mathbb{R}$. 

The following spectral gap estimate is due to Caputo. Recall that \assumptionV is in force.
\begin{lemma}[{Caputo~\cite[Corollary 5.2]{bib:caputo}}]
  \label{lem:spectral gap}
  There exists a constant $C > 0$ for which
  \[ \tmop{var}_{\nu^{\ell}_{\rho}} (F) \leqslant C \ell^2
     E_{\nu^{\ell}_{\rho}} \Big[ \frac{1}{2} \sum_{i = 0}^{\ell - 2}
     (\partial_i F - \partial_{i + 1} F)^2 \Big] \]
  for all smooth and bounded $F : \mathbb{R}^{\mathbb{Z}} \rightarrow
  \mathbb{R}$ that only depend on $(u(0), \ldots, u(\ell - 1))$, all $\ell
  \in \mathbb{N}$, and all $\rho \in \mathbb{R}$.
\end{lemma}

\begin{lemma}
  Let $F \in \bigcap_{\lambda \in \mathbb{R}} L^1 ({\mu}_{\lambda})$
  depend only on $u(0), \ldots, u(\ell - 1)$ with
  $\mathbb{E}_{{\mu}_{\lambda}} [F] = 0$ for all $\lambda \in
  \mathbb{R}$. Then $E_{\lambda} [F \varphi (\sum_{i = 0}^{\ell - 1} u(i))] =
  0$ for all $\varphi$ for which the right hand side is well defined and for
  all $\lambda$.
\end{lemma}

\begin{proof}
  Fix $\lambda$ and let $\lambda' \in \mathbb{R}$. We have
  \[ 0 = E_{\lambda + \lambda'} [F] = E_{\lambda} [F e^{\lambda' \left(
     \sum_{i = 0}^{\ell - 1} u(i) \right)}] \]
  for all $\lambda'$. Thus, the Laplace transform of the signed measure
  \[
     \sigma\Big(\sum_{i=0}^{\ell -1} u(i) \Big) \ni A \mapsto E_{\lambda} [F \1_A] \in \R
  \]
  is identically 0, which proves the claim.
\end{proof}

As a simple consequence of Lemma \ref{lem:spectral gap} we obtain a bound for
the $\CH^{- 1}$ norm. The proof is similar to the one in
\cite{bib:goncalvesJara}, but since we are in a different setting let us give
some details.

\begin{lemma}
  (see also \cite[Proposition~6]{bib:goncalvesJara})
  Let $C$ be the constant in
  Lemma~\ref{lem:spectral gap}. Then
  \[ \| F \|_{- 1}^2 \leqslant 2 C \ell^2 \langle F, F \rangle_{\lambda_0} \]
  for all $\ell \in \mathbb{N}$ and all $F \in \bigcap_{\lambda \in
  \mathbb{R}} L^1 ({\mu}_{\lambda}) \cap L^2 ({\mu}_{\lambda_0})$
  that depend only on $u(0), \ldots, u(\ell - 1)$ and that satisfy
  $E_{\lambda} [F] = 0$ for all $\lambda \in \mathbb{R}$.
\end{lemma}

\begin{proof}
  Consider a test function $G \in C^{\infty} (\mathbb{R}^{\mathbb{Z}},
  \mathbb{R})$ for the variational definition of the $\CH^{- 1}$ norm and
  define $G_{\ell} = E_{\lambda_0} [G|u(0), \ldots, u(\ell - 1)]$ and
  $\bar{G}_{\ell} = G_{\ell} - E_{\lambda_0} [G_{\ell} | \bar{u}_{\ell}]$ for
  $\bar{u}_{\ell} = \ell^{- 1} \sum_{i = 0}^{\ell - 1} u(i)$. Then we have
  \[ 2 \langle F, G \rangle_{\lambda_0} - \| G \|_1^2 = 2 \langle F, G_{\ell}
     \rangle_{\lambda_0} - \| G \|_1^2 = 2 \langle F, \bar{G}_{\ell}
     \rangle_{\lambda_0} - \| G \|_1^2, \]
  using the previous lemma in the second step. Now observe that
  \[ \| G \|_1^2 = \frac{1}{2} \sum_{i \in \mathbb{Z}} E_{\lambda_0}
     [(\partial_i G - \partial_{i + 1} G)^2] \geqslant \frac{1}{2} \sum_{i =
     0}^{\ell - 2} E_{\lambda_0} [(\partial_i G - \partial_{i + 1} G)^2], \]
  and that for $i \in [0, \ell - 1]$ the derivative $\partial_i (E_{\lambda_0}
  [G_{\ell} | \bar{u}_{\ell}])$ is independent of $i$ by exchangeability while
  $\partial_i G_{\ell} = E_{\lambda_0} [\partial_i G|u(0), \ldots, u(\ell -
  1)]$. Therefore, Jensen's inequality gives
  \[ \| G \|_1^2 \geqslant \frac{1}{2} \sum_{i = 0}^{\ell - 2} E_{\lambda_0} [(\partial_i
     \bar{G}_{\ell} - \partial_{i + 1} \bar{G}_{\ell})^2] =  \frac{1}{2}\sum_{i = 0}^{\ell
     - 2} E_{\lambda_0} [E_{\nu^{\ell}_{\bar{u}^\ell}} [(\partial_i \bar{G}_{\ell} -
     \partial_{i + 1} \bar{G}_{\ell})^2]] \geqslant  \frac{1}{2} C^{- 1} \ell^{- 2}
     E_{\lambda_0} [\tmop{var}_{\nu^{\ell}_{\bar{u}^\ell}} (\bar{G}_{\ell})], \]
  where the last step follows from
  Lemma~\ref{lem:spectral gap}.
  The definition of $\bar{G}_{\ell}$ implies that $E_{\nu^{\ell}_{\rho}}
  [\bar{G}_{\ell}] = 0$ for all $\rho$, and therefore we end up with
  \[ 2 \langle F, G \rangle_{\lambda_0} - \| G \|_1^2 \leqslant 2 (\langle F,
     F \rangle_{\lambda_0})^{1 / 2} (\langle \bar{G}_{\ell}, \bar{G}_{\ell}
     \rangle_{\lambda_0})^{1 / 2} -  \frac{1}{2} C^{- 1} \ell^{- 2} \langle \bar{G}_{\ell},
     \bar{G}_{\ell} \rangle_{\lambda_0}, \]
  from where we readily deduce the claimed bound for $\| F \|_{- 1}$.
\end{proof}

Now we obtain the following Corollary using exactly the same arguments as in~\cite{bib:goncalvesJara}.

\begin{corollary}
  \label{cor:GJ corollary 1}(see also \cite[Corollary~1]{bib:goncalvesJara})
  Let $m \in \N$, $0 \leqslant k_0 < \ldots < k_m $ and $F_1, \ldots, F_m \in L^2( [0, T]
  ; L^2(\mu_{\lambda_0}))$ be such $F_i
  (t, \cdummy) \in \bigcap_{\lambda \in \mathbb{R}} L^1 ({\mu}_{\lambda})$ depends only on $u(j)$ for $k_i \leqslant j
  \leqslant k_{i + 1} - 1$ and such that $E_{\lambda} [F_i (t, \cdummy)] = 0$
  for all $\lambda \in \mathbb{R}$, $i \in \{ 1, \ldots, m \}$, $t \in [0,
  T]$. Set $\ell_i = k_i - k_{i - 1}$. Then
  \[ E_{\lambda_0} \Big[ \sup_{t \in [0, T]} \Big( \int_0^t (F_1 + \cdots +
     F_m) (s, u^n_s) \mathd s \Big)^2 \Big] \lesssim \sum_{i = 1}^m
     \frac{\ell_i^2}{n^2} \int_0^T \langle F_i (s, \cdummy), F_i (s, \cdummy)
     \rangle_{\lambda_0} \mathd s. \]
\end{corollary}

\subsection{Equivalence of ensembles}\label{sec:equivEnsem}

Here we present a second order equivalence of ensembles result for the stationary measure $\mu_{\lambda_0}$. \assumptionV guarantees that the local limit
theorem holds uniformly in the parameter $\lambda$. To state this, we need
some more notation. If $(U_i)_{0 = 1}^{N - 1}$ is an i.i.d. family of random
variables distributed according to $p_{\lambda}$, then we write
$f^N_{\lambda}$ for the density of
\[ \frac{1}{\sqrt{N \rho'' (\lambda)}} \sum_{i = 0}^{N - 1} (U_i - \rho'
   (\lambda)) \nocomma. \]
We then have the following uniform local limit theorem:

\begin{lemma}[{\cite[Theorem~2.1]{bib:caputo}}]\label{lem:LLT}
  The potential $V$ is such that the local limit theorem holds uniformly in
  $\lambda$. More precisely,
  \[ | f^N_{\lambda} (u) - r_{\lambda, N} (u) | \lesssim N^{- 3 / 2}, \]
  uniformly in $u, \lambda$, where $r_{\lambda, N} = r^0 + N^{- 1 / 2}
  r_{\lambda}^1 + N^{- 1} r_{\lambda}^2$ with
  \[ r^0 (u) = \frac{1}{\sqrt{2 \pi}} e^{- u^2 / 2}, \hspace{2em}
     r^1_{\lambda} (u) = r^0 (u) \frac{m_{3, \lambda}}{6 \sigma^3_{\lambda}}
     H_3 (u), \]
  \[ r^2_{\lambda} (u) = r^0 (u) \left( \frac{m_{4, \lambda} - 3
     \sigma_{\lambda}^4}{24 \sigma_{\lambda}^4} H_4 (u) + \frac{m_{3,
     \lambda}^2}{72 \sigma^6_{\lambda}} H_6 (u) \right), \]
  and where $H_k$ denotes the $k$-th Hermite polynomial.
\end{lemma}

\begin{lemma}[{\cite[Lemma~2.2 and Lemma~2.4]{bib:caputo}}]
  \label{lem:UB}
  For all $k \in \mathbb{N}$ we have
  \[ \sup_{\lambda \in \mathbb{R}} \left( \frac{| m_{k, \lambda}
     |}{\sigma_{\lambda}^k} + \sigma^2_{\lambda} +
     \sigma^{- 2}_{\lambda} \right) < \infty . \]
\end{lemma}

With the help of these two lemmas we can derive the following bound. The proof is elementary but tedious and can be found in Appendix~\ref{app:equiv-ensem-proof}.

\begin{proposition}[Second order equivalence of ensembles]\label{prop:equiv-ensem}

Let $\ell \leqslant N / 2$ and let $F \in L^2 ({\mu}_{h' (\rho)})$ depend only on $u(0), \ldots, u(\ell - 1)$. Write
  \[
     \bar{u}^N \assign \frac{1}{N} u^N \assign \frac{1}{N} \sum_{k = 0}^{N - 1} u(k), \hspace{2em} \psi_F (N, \rho) \assign E_{\lambda_0} [F| \bar{u}^N = \rho], \hspace{2em}  \varphi_F (\rho) \assign E_{h' (\rho)} [F].
  \]
  Then
  \begin{align*}
     &E_{\lambda_0} \Big[ \Big| \psi_F (N, \bar{u}^N) - \varphi_F (\rho'
     (\lambda_0)) - \partial_{\rho} \varphi_F (\rho' (\lambda_0)) (
     \bar{u}^N - \rho' (\lambda_0) ) - \frac{1}{2} \partial_{\rho \rho}
     \varphi_F (\rho' (\lambda_0)) \big( (\bar{u}^N - \rho' (\lambda_0))^2 -
     \frac{\sigma_{\lambda_0}^2}{N} \big) \Big|^2 \Big] \\
     & \hspace{50pt} \lesssim \left( \frac{\ell}{N} \right)^3 \sup_{\lambda}
     \tmop{var}_{\lambda} (F).
  \end{align*}
\end{proposition}

\begin{remark}
   We will later apply this with $F(u) = V'(u(0))$ for which an integration by parts yields
   \[
      E_\lambda[F] = \int_\R V'(u) p_{\lambda}(u) \dd u = \int_\R ( V'(u) - \lambda) p_{\lambda}(u) \dd u + \lambda = - \int_\R \partial_u p_{\lambda}(u) \dd u + \lambda =\lambda
   \]
   and then
   \[
      \mathrm{var}_\lambda(F) = \int_\R (V'(u) - \lambda)^2 p_{\lambda}(u) \dd u = \int_\R ( V'(u) - \lambda)(- \partial_u p_{\lambda}(u)) \dd u  =  \int_\R V''(u) p_{\lambda}(u) \dd u.
   \]
   We assumed that $V'$ is Lipschitz-continuous, so the supremum in $\lambda$ of the right hand side is finite. If however $V''$ was unbounded, then the supremum in $\lambda$ of the right hand side would certainly be infinite: we know from Lemma~\ref{lem:UB} that the variance of $u(0)$ stays uniformly bounded in $\lambda$ while by varying $\lambda$ we can achieve any mean $\rho'(\lambda)$ for $u(0)$ and in particular we can send $\rho'(\lambda)$ to those regions where $V''$ is very large. So if we wanted to deal with non-Lipschitz continuous $V'$, then we would need to be more careful in the estimates leading to Proposition~\ref{prop:equiv-ensem}.
\end{remark}

\subsection{Derivation of the second order Boltzmann-Gibbs principle}\label{sec:BG}

We can now combine Corollary~\ref{cor:GJ corollary 1} and
Proposition~\ref{prop:equiv-ensem} to derive the second order Boltzmann-Gibbs
principle. As Gon{\c c}alves and Jara point out in \cite{bib:goncalvesJara},
these two corollaries are the only ingredients needed to make their proof of
the second order Boltzmann-Gibbs principle work. And although we are in a
different setting we can indeed proceed by building on the same lemmas as they do, and prove those lemmas using the same arguments provided in
\cite{bib:goncalvesJara}. So we will omit most of the proofs and simply include
references to the corresponding results in that work.
Let us introduce the notation
\[ (\tau_k u) (\ell) = u (k + \ell) \]
for all $u \in \mathbb{R}^{\mathbb{Z}}$, $k, \ell \in \mathbb{Z}$. If $F :
\mathbb{R}^{\mathbb{Z}} \rightarrow \mathbb{R}$ is a local function, we
write $\tau_k F (u) = F (\tau_k u)$.

\begin{lemma}[One block estimate, see also~{\cite[Lemma~1]{bib:goncalvesJara}}]
  \label{lem:one block}
  Let $F \in L^2 ({\mu}_{\lambda})$ be a local function that depends only
  on $u (0), \ldots, u (\ell_0 - 1)$ and let $g \in L^2 ([0, T] \times
  \mathbb{Z})$. Then for all $\lambda \in \mathbb{R}$
  \[ \mathbb{E}_{{\mu}_{\lambda}} \Big[ \Big( \int_0^T \sum_k \tau_k (F
     (u^n_s) - \psi_F (\ell_0, \overline{(u^n_s)}^{\ell_0})) g_s (k) \mathd s
     \Big)^2 \Big] \lesssim \frac{\ell_0^3}{n^2} \| g \|_{L^2 ([0, T]
     \times \mathbb{Z})}^2 \tmop{var}_{\lambda} (F) . \]
\end{lemma}

\begin{remark}
  Here we will only apply this with $\ell_0 = 1$ for which the left
  hand side simply vanishes. However, the lemma will allow us to prove the
  second order Boltzmann-Gibbs principle for general $\ell_0$, which may be a
  useful result in itself.
\end{remark}

\begin{lemma}[Renormalization step, see also~{\cite[Lemma~2]{bib:goncalvesJara}}]
  Let $F \in \bigcap_{\lambda \in \mathbb{R}} L^2 ({\mu}_{\lambda})$
  be a local function that depends only on $u (0), \ldots, u (\ell_0 - 1)$ and
  let $g \in L^2 ([0, T] \times \mathbb{Z})$. Then for all $\ell \geqslant
  \ell_0$
  \[ \mathbb{E}_{{\mu}_{\lambda_0}} \Big[ \Big( \int_0^T \sum_k \tau_k
     \big( \psi_F ( \ell, \overline{(u^{n_{}}_s)}^{\ell} ) -
     \psi_F ( 2 \ell, \overline{(u^{n_{}}_s)}^{2 \ell} ) \big)
     g_s (k) \mathd s \Big)^2 \Big] \lesssim_{\ell_0}
     \frac{\ell^{\beta}}{n^2} \| g \|_{L^2 ([0, T] \times \mathbb{Z})}^2
     \sup_{\lambda} \tmop{var}_{\lambda} (F), \]
  where
  \begin{enumerateroman}
    \item $\beta = 2$ if $\partial_{\rho} \varphi_F |_{\rho = h' (\lambda_0)}
    \neq 0$,
    
    \item $\beta = 1$ if $\partial_{\rho} \varphi_F |_{\rho = h' (\lambda_0)}
    = 0$,
    
    \item $\beta = 0$ if $\partial_{\rho} \varphi_F |_{\rho = h' (\lambda_0)}
    = \partial_{\rho \rho} \varphi_F |_{\rho = h' (\lambda_0)} = 0$.
  \end{enumerateroman}
\end{lemma}

\begin{lemma}[Two blocks estimate, see also~{\cite[Lemma~3]{bib:goncalvesJara}}]
  \label{lem:two block}
  Let $F \in \bigcap_{\lambda \in \mathbb{R}} L^2 ({\mu}_{\lambda})$ be a
  local function that depends only on $u (0), \ldots, u (\ell_0 - 1)$ and let
  $g \in L^2 ([0, T] \times \mathbb{Z})$. Then for any $\ell \geqslant
  \ell_0$
  \[ \mathbb{E}_{{\mu}_{\lambda_0}} \Big[ \Big( \int_0^T \sum_k \tau_k
     \big( \psi_F ( \ell_0, \overline{(u^{n_{}}_s)}^{\ell_0} ) -
     \psi_F ( \ell, \overline{(u^{n_{}}_s)}^{\ell} ) \big) g_s
     (k) \mathd s \Big)^2 \Big] \lesssim_{\ell_0} \frac{\beta_{\ell}}{n^2}
     \| g \|_{L^2 ([0, T] \times \mathbb{Z})}^2 \sup_{\lambda}
     \tmop{var}_{\lambda} (F), \]
  where
  \begin{enumerateroman}
    \item $\beta_{\ell} = \ell^2$ if $\partial_{\rho} \varphi_F |_{\rho = h'
    (\lambda_0)} \neq 0$,
    
    \item $\beta_{\ell} = \ell$ if $\partial_{\rho} \varphi_F |_{\rho = h'
    (\lambda_0)} = 0$,
    
    \item $\beta_{\ell} = (\log \ell)^2$ if $\partial_{\rho} \varphi_F |_{\rho
    = h' (\lambda_0)} = \partial_{\rho \rho} \varphi_F |_{\rho = h'
    (\lambda_0)} = 0$.
  \end{enumerateroman}
\end{lemma}

The next Lemma~\ref{lem:linear} has no equivalent result in \cite{bib:goncalvesJara}, but Lemma~\ref{lem:quadratic-corrector} does and is a second order version of Lemma~\ref{lem:linear} which follows from the same proof.

\begin{lemma}
  \label{lem:linear}Let $F \in \bigcap_{\lambda \in \mathbb{R}} L^2
  ({\mu}_{\lambda})$ be a local function that depends only on $u (0),
  \ldots, u (\ell_0 - 1)$ and assume that $\varphi_F (\rho' (\lambda_0)) = 0$.
  Let $g \in L^2 ([0, T] \times \mathbb{Z})$. Then
  \begin{align*}
     &\mathbb{E}_{{\mu}_{\lambda_0}} \Big[ \Big( \int_0^T \sum_k \tau_k
     \big( \psi_F ( \ell, \overline{(u^{n_{}}_s)}^{\ell} ) -
     \partial_{\rho} \varphi_F (\rho' (\lambda_0)) (\overline{(u^n_s)}^{\ell}
     - \rho' (\lambda_0)) \big) g_s (k) \mathd s \Big)^2 \Big] \\
     & \hspace{50pt} \lesssim_{\ell_0} \frac{T}{\ell} \| g \|_{L^2 ([0, T] \times
     \mathbb{Z})}^2 \sup_{\lambda} \tmop{var}_{\lambda} (F) 
  \end{align*}
\end{lemma}

\begin{proof}
  We apply twice the Cauchy-Schwarz inequality to obtain
  \begin{align*}
   &\mathbb{E}_{{\mu}_{\lambda_0}} \Big[ \Big( \int_0^T \sum_k \tau_k
   \big( \psi_F ( \ell, \overline{(u^{n_{}}_s)}^{\ell} ) -
   \partial_{\rho} \varphi_F (\rho' (\lambda_0)) (\overline{(u^n_s)}^{\ell}
   - \rho' (\lambda_0)) \big) g_s (k) \mathd s \Big)^2 \Big] \\
   &\qquad \leqslant T \int_0^T \mathbb{E}_{{\mu}_{\lambda_0}} \Big[ \Big(
   \sum_{j = 0}^{\ell - 1} \sum_i \tau_{i \ell + j} \big( \psi_F (
   \ell, \overline{(u^{n_{}}_s)}^{\ell} ) - \partial_{\rho} \varphi_F
   (\rho' (\lambda_0)) (\overline{(u^n_s)}^{\ell} - \rho' (\lambda_0))
   \big) g_s (i \ell + j) \Big)^2 \Big] \mathd s \\
   &\qquad \leqslant T \ell \int_0^T \sum_{j = 0}^{\ell - 1}
   \mathbb{E}_{{\mu}_{\lambda_0}} \Big[ \Big( \sum_i \tau_{i \ell +
   j} \big( \psi_F ( \ell, \overline{(u^{n_{}}_s)}^{\ell} ) -
   \partial_{\rho} \varphi_F (\rho' (\lambda_0)) (\overline{(u^n_s)}^{\ell}
   - \rho' (\lambda_0)) \big) v_s (i \ell + j) \Big)^2 \Big] \mathd s.
  \end{align*}
  Now obserge that $\mathbb{E}_{{\mu}_{\lambda_0}} [ \psi_F (
  \ell, \overline{(u^{n_{}}_s)}^{\ell} ) ]
  =\mathbb{E}_{{\mu}_{\lambda_0}} [\partial_{\rho} \varphi_F (\rho'
  (\lambda_0)) (\overline{(u^n_s)}^{\ell} - \rho' (\lambda_0))] = 0$, and
  therefore
  \begin{align*}
   &T \ell \int_0^T \sum_{j = 0}^{\ell - 1}
   \mathbb{E}_{{\mu}_{\lambda_0}} \Big[ \Big( \sum_i \tau_{i \ell +
   j} \big( \psi_F ( \ell, \overline{(u^{n_{}}_s)}^{\ell} ) -
   \partial_{\rho} \varphi_F (\rho' (\lambda_0)) (\overline{(u^n_s)}^{\ell}
   - \rho' (\lambda_0)) \big) g_s (i \ell + j) \Big)^2 \Big] \mathd s \\
   &\qquad = T \ell \int_0^T \sum_{j = 0}^{\ell - 1} \sum_i
   \mathbb{E}_{{\mu}_{\lambda_0}} \Big[ \Big( \tau_{i \ell + j}
   \big( \psi_F ( \ell, \overline{(u^{n_{}}_s)}^{\ell} ) -
   \partial_{\rho} \varphi_F (\rho' (\lambda_0)) (\overline{(u^n_s)}^{\ell}
   - \rho' (\lambda_0)) \big) g_s (i \ell + j) \Big)^2 \Big] \mathd s \\
   &\qquad = T \ell \int_0^T \sum_k \mathbb{E}_{{\mu}_{\lambda_0}} \Big[
   \Big( \psi_F ( \ell, \overline{(u^{n_{}}_s)}^{\ell} ) -
   \partial_{\rho} \varphi_F (\rho' (\lambda_0)) (\overline{(u^n_s)}^{\ell}
   - \rho' (\lambda_0)) \Big)^2 \Big] g^2_s (k) \mathd s \\
   &\qquad = T \ell \| g \|_{L^2 ([0, T] \times \mathbb{Z})}^2 E_{\lambda_0}
   [(\psi_F (\ell, \bar{u}^{\ell}) - \partial_{\rho} \varphi_F (\rho'
   (\lambda_0)) (\bar{u}^{\ell} - \rho' (\lambda_0)))^2],
  \end{align*}
  and by Proposition~\ref{prop:equiv-ensem} the expectation on the right
  hand side is bounded by
  \[ E_{\lambda_0} [((\psi_F (\ell, \bar{u}^{\ell}) - \partial_{\rho}
     \varphi_F (\rho' (\lambda_0)) (\bar{u}^{\ell} - \rho' (\lambda_0))))^2]
     \lesssim \frac{\ell_0^3}{\ell^2} \sup_{\lambda} \tmop{var}_{\lambda} (F),
  \]
  from where our claim follows. Actually we could even obtain $\ell_0^2$ on
  the right hand side, but we do not care about the dependence on $\ell_0$.
\end{proof}

\begin{lemma}(see also \cite[Lemma~4]{bib:goncalvesJara})\label{lem:quadratic-corrector}
  Let $F \in \bigcap_{\lambda \in \mathbb{R}} L^2 ({\mu}_{\lambda})$ be a
  local function that depends only on $u (0), \ldots, u (\ell_0 - 1)$ and
  assume that $\varphi_F (\rho' (\lambda_0)) = (\partial_{\rho} \varphi_F)
  (\rho' (\lambda_0)) = 0$. Define for $\ell \in \mathbb{N}$ and $\rho \in
  \mathbb{R}$
  \[ \CQ_{\lambda_0} (\ell ; u) = (\bar{u}^{\ell} - \rho' (\lambda_0))^2 -
     \frac{\sigma^2_{\lambda_0}}{\ell} . \]
  Then we have for any $g \in L^2 ([0, T] \times \mathbb{Z})$
  \begin{align*}
  & \mathbb{E}_{{\mu}_{\lambda_0}} \Big[ \Big( \int_0^T \sum_k \tau_k
     \big( \psi_F ( \ell, \overline{(u^{n_{}}_s)}^{\ell} ) -
     \frac{(\partial_{\rho \rho} \varphi_F) (\rho' (\lambda_0))}{2}
     \CQ_{\lambda_0} (\ell ; u^n_s) \big) g_s (k) \mathd s \Big)^2 \Big]
  \\
  & \hspace{50pt} \lesssim_{\ell_0} \frac{T}{\ell^2} \| g \|_{L^2 ([0, T] \times
     \mathbb{Z})}^2 \sup_{\lambda} \tmop{var}_{\lambda} (F) .
  \end{align*}
\end{lemma}

The next theorem is not contained in \cite{bib:goncalvesJara},
but it is an easy consequence of the lemmas that we established so far.

\begin{theorem}[Boltzmann-Gibbs principle]
  \label{thm:first order BG}
  Let $F \in \bigcap_{\lambda \in \mathbb{R}} L^2 ({\mu}_{\lambda})$ be a
  local function that depends only on $u (0), \ldots, u (\ell_0 - 1)$ and let
  $g \in L^2 ([0, T] \times \mathbb{Z})$. Then for all $\ell \geqslant
  \ell_0$
  \begin{align*}
     &\mathbb{E}_{{\mu}_{\lambda_0}} \Big[ \Big( \int_0^T \sum_k \tau_k
     \big( F (u^n_s) - \varphi_F (\rho' (\lambda_0)) - \partial_{\rho}
     \varphi_F ( \rho' (\lambda_0) \big) (u_s^n (0) - \rho'
     (\lambda_0)) ) g_s (k) \mathd s \Big)^2 \Big] \\
     & \hspace{50pt} \lesssim_{\ell_0} \left( \frac{\ell}{n^2} + \frac{T}{\ell} \right) \| g
     \|_{L^2 ([0, T] \times \mathbb{Z})}^2 \sup_{\lambda}
     \tmop{var}_{\lambda} (F) . 
  \end{align*}
\end{theorem}

\begin{proof}
  {\color{black} Let $\hat{F} (u) \assign F (u) - \varphi_F (\rho'
  (\lambda_0)) - \partial_{\rho} \varphi_F (\rho' (\lambda_0)) (u (0) - \rho'
  (\lambda_0))$, so that for any $\ell \in \mathbb{N}$
  \begin{align*}
   &\mathbb{E}_{{\mu}_{\lambda_0}} \left[ \left( \int_0^T \sum_k \tau_k
   \left( F (u^n_s) - \varphi_F (\rho' (\lambda_0)) - \partial_{\rho}
   \varphi_F \left( \rho' (\lambda_0) \right) (u_s^n (0) - \rho'
   (\lambda_0)) \right) g_s (k) \mathd s \right)^2 \right] \\
   &\qquad=\mathbb{E}_{{\mu}_{\lambda_0}} \left[ \left( \int_0^T \sum_k \tau_k
   \hat{F} (u^n_s) g_s (k) \mathd s \right)^2 \right] \\
   &\qquad=\mathbb{E}_{{\mu}_{\lambda_0}} \left[ \left( \int_0^T \sum_k \tau_k
   (\hat{F} (u^n_s) - \partial_{\rho} \varphi_{\hat{F}} (\rho' (\lambda_0))
   (\overline{(u_s^n)}^{\ell} - \rho' (\lambda_0))) g_s (k) \mathd s
   \right)^2 \right].
  \end{align*}
  Now we combine Lemma~\ref{lem:one block} with Lemma~\ref{lem:two block} and
  Lemma~\ref{lem:linear} to get for $\ell \geqslant \ell_0$
  \begin{align*}
   &\mathbb{E}_{{\mu}_{\lambda_0}} \left[ \left( \int_0^T \sum_k \tau_k
   (\hat{F} (u^n_s) - \partial_{\rho} \varphi_{\hat{F}} (\rho' (\lambda_0))
   (\overline{(u_s^n)}^{\ell} - \rho' (\lambda_0))) g_s (k) \mathd s
   \right)^2 \right] \\
   &\quad
   \lesssim \mathbb{E}_{{\mu}_{\lambda_0}} \left[ \left( \int_0^T
   \sum_k \tau_k (\hat{F} (u^n_s) - \psi_{\hat{F}} (\ell_0,
   \overline{(u^n_s)}^{\ell_0})) g_s (k) \mathd s \right)^2 \right] \\
   &\qquad
   +\mathbb{E}_{{\mu}_{\lambda_0}} \left[ \left( \int_0^T \sum_k \tau_k
   \big( \psi_{\hat{F}} (\ell_0, \overline{(u^n_s)}^{\ell_0}) -
   \psi_{\hat{F}} ( \ell, \overline{(u^{n_{}}_s)}^{\ell} )
   \big) g_s (k) \mathd s \right)^2 \right] \\
   &\qquad
   +\mathbb{E}_{{\mu}_{\lambda_0}} \left[ \left( \int_0^T \sum_k \tau_k
   \big( \psi_{\hat{F}} ( \ell, \overline{(u^{n_{}}_s)}^{\ell} )
   - \partial_{\rho} \varphi_{\hat{F}} (\rho' (\lambda_0))
   (\overline{(u_s^n)}^{\ell} - \rho' (\lambda_0)) \big) g_s (k) \mathd s
   \right)^2 \right] \\
   &\quad
   \lesssim_{\ell_0} \frac{1}{n^2} \| g \|_{L^2 ([0, T] \times
   \mathbb{Z})}^2 \tmop{var}_{\lambda_0} (\hat{F}) + \frac{\ell}{n^2} \| g
   \|_{L^2 ([0, T] \times \mathbb{Z})}^2 \sup_{\lambda}
   \tmop{var}_{\lambda} (\hat{F}) + \frac{T}{\ell} \| g \|_{L^2 ([0, T]
   \times \mathbb{Z})}^2 \sup_{\lambda} \tmop{var}_{\lambda} (\hat{F}).
  \end{align*}
  It now suffices to note that
  \[ \tmop{var}_{\lambda} (\hat{F}) \lesssim \tmop{var}_{\lambda} (F) +
     \tmop{var}_{\lambda} (\partial_{\rho} \varphi_F (\rho' (\lambda_0)) (u
     (0) - \rho' (\lambda_0))) = \tmop{var}_{\lambda} (F) + (\partial_{\rho}
     \varphi_F (\rho' (\lambda_0)))^2 \sigma_{\lambda}^2 \]
  and finally
  \[ | \partial_{\rho} \varphi_F (\rho' (\lambda_0)) | = | E_{\lambda_0}
     [(u^{\ell_0} - \ell_0 \rho' (\lambda_0)) F] h'' (\rho' (\lambda_0)) | =
     \sigma_{\lambda_0}^{- 2} | E_{\lambda} [(u^{\ell_0} - \ell_0 \rho'
     (\lambda_0)) F] | \]
  \[ \lesssim_{\ell_0} \sigma_{\lambda_0}^{- 2} (\tmop{var}_{\lambda_0}
     (F))^{1 / 2} \sigma_{\lambda_0}^2, \]
     from where an application of Lemma \ref{lem:UB} proves
  our claim.}
\end{proof}

\begin{theorem}[Second order Boltzmann-Gibbs principle, see also~{\cite[Theorem~7]{bib:goncalvesJara}}]
  \label{thm:second order BG}  
  Let $F$ be a local function that depends only on $u (0), \ldots, u (\ell_0 -
  1)$ and assume that $\varphi_F (\rho' (\lambda_0)) = (\partial_{\rho}
  \varphi_F) (\rho' (\lambda_0)) = 0$. Let $g \in L^2 ([0, T] \times
  \mathbb{Z})$. Then, for any $\ell \geqslant \ell_0$,
\end{theorem}
\begin{align*}
   &\mathbb{E}_{{\mu}_{\lambda_0}} \left[ \left( \int_0^T \sum_k \tau_k
   \left( F (u^n_s) - \frac{(\partial_{\rho \rho} \varphi_F) (\rho'
   (\lambda_0))}{2} \CQ_{\lambda_0} (\ell ; u^n_s) \right) g_s (k) \mathd s
   \right)^2 \right] \\
   & \hspace{50pt} \lesssim_{\ell_0} \left( \frac{\ell}{n^2} + \frac{T}{\ell^2} \right) \| g
   \|_{L^2 ([0, T] \times \mathbb{Z})}^2 \sup_{\lambda} \tmop{var}_{\lambda}
   (F) .
\end{align*}
The proof is similar to the one of Theorem~\ref{thm:first order BG} but
easier, because here we subtract a quadratic term that depends on the local
average $\overline{(u^n_s)}^{\ell}$, whereas in Theorem~\ref{thm:first order
BG} we still had to replace the linear term $\partial_{\rho} \varphi_F (\rho'
(\lambda)) (\overline{(u^n_s)}^{\ell} - \rho' (\lambda))$ of
Lemma~\ref{lem:linear} by $\partial_{\rho} \varphi_F (\rho' (\lambda)) (u_s^n
(0) - \rho' (\lambda))$.

\section{Proof of convergence}\label{sec:convergence}

We now have all the tools needed to prove that the rescaled fluctuations of $u$ converge to the solution of the stochastic Burgers equation. We start by showing tightness, which as usual with the martingale problem approach is the hardest part.

\subsection{Tightness}\label{sec:tightness}

Recall that we defined
\[ v^n_t = \sum_k n^{1 / 2} (u_{n^2 t} (k) - \rho' (\lambda_0)) n^{- 1}
   \delta_{n^{- 1} k + c_n t} \]
with $c_n =  n^{1 / 2} \partial_{\rho} \varphi_{V'} (\rho' (\lambda_0))$.
Thus, for any test function $\eta \in \CS$
\[ v^n_t (\eta) = \sum_k n^{- 1 / 2} (u_{n^2 t} (k) - \rho' (\lambda_0)) \eta
   (n^{- 1} k + c_n t) = \sum_k n^{- 1 / 2} (u^n_t (k) - \rho' (\lambda_0))
   \eta (n^{- 1} k + c_n t), \]
where we recall that $u^n_t = u_{n^2 t}$, and then
\begin{align*}
 \mathd v^n_t (\eta)
 &= \sum_k n^{- 1 / 2} \eta (n^{- 1} k + c_n t) \mathd
 u^n_t (k) + \sum_k n^{- 1 / 2} (u^n_t (k) - \rho' (\lambda_0)) \mathd \eta
 (n^{- 1} k + c_n t) \\
 &= \sum_k n^{- 1 / 2} \eta (n^{- 1} k + c_n t) \left( \CL^{(n)} u^n_t (k)
 \mathd t + n^{- 1 / 2} \mathd \nabla_n^{(1)} W^n_t (k) \right) + c_n v^n_t
 (\nabla \eta) \mathd t,
\end{align*}
with $\mathd [W^n (i), W^n (j)]_t = n \delta_{i, j} \mathd t$. We can further
decompose the drift term into two parts,
\begin{align*}
 &\sum_k n^{- 1 / 2} \eta (n^{- 1} k + c_n t) \CL^{(n)} u^n_t (k) \mathd t +
 c_n v^n_t (\nabla \eta) \mathd t \\
 &\quad= \sum_k n^{- 1 / 2} \eta (n^{- 1} k + c_n t) \CL^{(n)}_S u^n_t (k) \mathd
 t
 + \left( \sum_k n^{- 1 / 2} \eta (n^{- 1} k + c_n t) \CL^{(n)}_A u^n_t (k)
 + c_n v^n_t (\nabla \eta) \right) \mathd t \\
 &\quad= : \mathd S^n_t (\eta) + \mathd A^n_t (\eta).
\end{align*}
So overall we obtain a decomposition into symmetric part, asymmetric part, and
martingale part,
\[ v^n_t (\eta) - v^n_0 (\eta) = S^n_t (\eta) + A^n_t (\eta) + M^n_t (\eta),
\]
with
\[ M^n_t (\eta) : = \int_0^t \sum_k n^{- 1} \eta (n^{- 1} k + c_n r) \mathd
   \nabla^{(1)}_n W^n_r (k) . \]
We will show the joint tightness of all the different contributions, which
will then make it easy to identify the limit of $v^n$.

\begin{lemma}
  \label{lem:tightness}The family $(v^n_0, S^n, A^n, M^n)_{n \in \mathbb{N}}$
  is tight in $\CS' \times C \left( [0, 1], \CS' \right)^3$, and for all
  fixed times $t \geqslant 0$ the random variables $(v^n_t)_{n \in
  \mathbb{N}}$ converges in distribution in $\CS'$ to a spatial white noise
  with variance $\sigma_{\lambda_0}^2$. In particular, $(v^n)_{n \in
    \mathbb{N}}$ is tight in $C \left( [0, 1], \CS' \right)$.
\end{lemma}

Before we get to the proof, let us link the constants $\partial_\rho \varphi_{V'}(\rho)$ that appear in the formulation of the second order Boltzmann-Gibbs principle to the centered moments of $u(t,i)$ under $\mu_{\lambda_0}$.

\begin{lemma}\label{lem:constants}
   We have $\varphi_{V'}(\rho) = h'(\rho)$, and therefore in particular
   \[
      \partial_{\rho} \varphi_{V'}(\rho) = \sigma_{h'(\rho)}^{-2}, \qquad \partial_{\rho\rho} \varphi_{V'}(\rho) = -\frac{m_{3,h'(\rho)}}{\sigma_{h'(\rho)}^6}.
   \]
\end{lemma}

\begin{proof}
   It suffices to integrate by parts:
   \begin{align*}
      \varphi_{V'}(\rho) & = \int_\R V'(u) p_{h'(\rho)}(u) \dd u = \int_\R ( V'(u) - h'(\rho)) p_{h'(\rho)}(u) \dd u + h'(\rho) \\
      & = - \int_\R \partial_u p_{h'(\rho)}(u) \dd u + h'(\rho) = h'(\rho).
   \end{align*}
\end{proof}

\subsubsection{Convergence at fixed times}\label{sec:fixedTimes}

First we have to check that $v^n_0$ is tight in $\CS'$ which turns out to be
quite simple. Indeed, we even show that for all fixed times $v^n_t$ is tight.
By Mitoma's criterion \cite{bib:mitoma} (which also holds for
$\CS'$-valued processes; a form in which we shall use it later on) it is
enough to verify that $v^n_t (\eta)$ is tight, for every $\eta \in \CS .$ We
show that $v^n_t (\eta)$ converges in law to $\sigma_{\lambda_0} \xi
(\eta)$ for all $t \geqslant 0$, where $\xi$ is a spatial white noise. Start by noting that
\[ v^n_t (\eta) = \sum_k n^{- 1 / 2} (u^n_t (k) - \rho' (\lambda_0)) \eta
   (n^{- 1} k + c_n t) . \]
We split this sum into two parts, depending on $n$. One part is shown to go to
zero and the other one is shown to converge to $\sigma_{\lambda_0} \xi (\eta)$
using the Lyapunov criterion for the central limit theorem. We recall:

\begin{lemma}[{\cite[Theorem 27.3]{bib:billingsley}}]
  For each $n$ assume we have a sequence of independent random variables $X_{n
  1}, \ldots, X_{n r_n}$ with zero mean and $\sigma^2_{n k} \assign
  \mathbb{E} [X^2_{n k}] < \infty$. Let $s_n^2 \assign \sum_{k = 1}^{r_n}
  \sigma^2_{n k}$ and assume there is $\delta > 0$ such that
  \[ \lim_{n \rightarrow \infty} \frac{1}{s_n^{2 + \delta}} \sum_{k = 1}^{r_n}
     \mathbb{E} [| X^{}_{n k} |^{2 + \delta}] = 0. \]
  Then $\frac{1}{s_n} \sum_{k = 1}^{r_n} X_{n k}$ converges in law to a
  standard normal distribution.
\end{lemma}
We apply this to $X_{n k} \assign n^{- 1 / 2} (u^n_t (k) - \rho' (\lambda_0))
\eta (n^{- 1} k + c_n t)$ and $r_n \assign 2 n^2 + 1$ (but summing from $- n^2$
to $n^2$ around $- c_n t$). Then
\begin{align*}
   s_n^2 & = n^{- 1} \sum_{| k + n c_n t | \leqslant n^2} | \eta (n^{- 1} (k + n
   c_n t)) |^2 E_{\lambda_0} [(u (k) - \rho' (\lambda_0))^2] \\
   & = n^{- 1} \sum_{| k + n c_n t | \leqslant n^2} | \eta (n^{- 1} (k + n c_n t)) |^2 \sigma^2_{\lambda_0} \xrightarrow{n \rightarrow \infty}  \int_\R |\eta (x)^2| \mathd x \sigma^2_{\lambda_0}.
\end{align*}
The Lyapunov condition is satisfied with $\delta = 1$, since $s^3_n$ is
bounded below and
\[ \sum_{| k + n c_n t | \leqslant n^2}^{} \mathbb{E} [| X^{}_{n k} |^{3}] = n^{- 3 / 2} \sum_{| k + n c_n t | \leqslant n^2} | \eta (n^{- 1}
   (k + n c_n t)) |^3 E_{\lambda_0} [(u (k) - \rho' (\lambda_0))^3] . \]
Now $n^{1 / 2}$ times the last expression converges to a constant times
$\int_{\R} |\eta (x)|^3 \mathd x$, so that the expression in fact converges to $0$
and the Lyapunov condition is satisfied.
We are left with the remainder
\begin{align*}
 &\sum_{| k + n c_n t | > n^2} n^{- 1 / 2} (u^n_t (k) - \rho' (\lambda_0))
 \eta (n^{- 1} (k + n c_n t)) \\
 &\quad \lesssim  \sum_{| k + c_n t | > n^2} n^{- 1 / 2} (u^n_t (k) - \rho'
 (\lambda_0)) \frac{1}{1 + | n^{- 1} (k + n c_n t) |},
\end{align*}
and then $\mathbb{E} [\lfloor \ldots |^2]$ is bounded by a constant times
\begin{align*}
 \frac1{n} \sum_{k : | k + n c_n t | > n^2} & \frac{1}{1 + | n^{- 1} (k + n c_n
 t) |^2}
 \lesssim \frac1{n} \sum_{\ell : | \ell | > n^2} \frac{1}{1 + | n^{- 1} \ell
 |^2} 
 \lesssim n \sum_{\ell : | \ell | > n^2} \frac{1}{1 + | \ell |^2} \lesssim \frac1{n},
\end{align*}
which goes to $0$ as $n \rightarrow \infty \nocomma$.

\subsubsection{The martingale term}\label{sec:martingale}

Next, let us show tightness and convergence of the martingale term, which has
quadratic variation
\begin{align*}
 \left[ \int_0^{\cdot} \sum_k n^{- 1} \eta (n^{- 1} k + c_n r) \mathd
 \nabla^{(1)}_n W^n_r (k) \right]_t
 &= \left[ \int_0^{\cdot} \sum_k n^{- 1}
 \nabla^{(1)}_n \eta (n^{- 1} k + c_n r) \mathd W^n_r (k) \right]_t \\
 &= n^{- 2} \sum_k \int_0^t | \nabla^{(1)}_n \eta (n^{- 1} k + c_n r) |^2 n \mathd r,
\end{align*}
where we used that $\mathd [W^n (i), W^n (j)]_t = n \delta_{i, j} \mathd t$. It is not hard to see that for fixed $r \in [0, t]$
\[ n^{- 1} \sum_k | \nabla^{(1)}_n \eta (n^{- 1} k + c_n r) |^2 \lesssim \|
   \nabla \eta \|_{H^1}^2, \]
where
$\| \varphi \|_{H^k} = \sum_{\ell \leqslant k} \| \varphi^{(\ell)} \|_{L^2
   (\mathbb{R})}$.
So an application of the Burkholder-Davis-Gundy inequality yields for all $p
\geqslant 1$ and $s < t$
\begin{align*}
 \mathbb{E}_{{\mu}_{\lambda_0}} [| M^n_t (\eta) - M^n_s (\eta) |^p]
 &\lesssim \left( n^{- 1} \sum_k \int_s^t | \nabla^{(1)}_n \eta (n^{- 1} k +
 c_n r) |^2 \mathd r \right)^{p / 2} \\
 &\lesssim | t - s |^{p / 2} \| \nabla \eta \|_{H^1}^p \leqslant | t - s |^{p
 / 2} \| \eta \|_{H^2}^p .
\end{align*}
Now we apply the Kolmogorov-Chentsov criterion 
\cite{bib:kolmogorovCentsov} to obtain the tightness of $M^n(\eta)$ in $C\left( [0,1], \R \right)$.
By Mitoma's criterion \cite{bib:mitoma} $(M^n)$ is hence tight
in $C \left( [0, 1], \CS' \right)$.
In fact, by the dominated convergence
theorem we also get that the quadratic variation converges to $t \| \nabla
\eta \|^2_{L^2}$, which together with \cite[Proposition 5.1]{bib:aldous}
shows that $(M^n)$ converges to a space-time white noise. But of course we
still need to verify that the limit is also a martingale in the filtration
generated by $v$, the limit of $(v^n)$.

\subsubsection{The symmetric term}\label{sec:symmetric}

Next, we have to deal with the symmetric contribution
\[ \mathd S^n_t (\eta) = \sum_k n^{- 1 / 2} \eta (n^{- 1} k + c_n t)
   \CL^{(n)}_S u^n_t(k) \mathd t = \sum_k n^{3 / 2} \eta (n^{- 1} k + c_n
   t) \CL_S u^n_t (k) \mathd t, \]
where
\[ \CL_S = \sum_j \frac{1}{2} \Delta_D [V' (u)] (j) \partial_j + \sum_j
   (\partial_{j, j}^2 - \partial_{j, j + 1}^2), \]
so that $\CL_S u^n_t (k) = (1 / 2) \Delta_D [V' (u^n_t)_{}] (k)$, which gives
\begin{align*}
 \mathd S^n_t (\eta)
 &= \frac{1}{2} \sum_k n^{- 1 / 2} (\Delta_n \eta) (n^{-
 1} k + c_n t) V' (u^n_t (k))_{} \mathd t \\
 &= \frac{1}{2} \sum_k n^{- 1 / 2} \Delta_n \eta (n^{- 1} k + c_n t) (V'
 (u^n_t (k)) - \varphi_{V'} (\rho' (\lambda_0))) \mathd t,
\end{align*}
where we recall that $\varphi_{V'} (\rho' (\lambda_0))
=\mathbb{E}_{{\mu}_{\lambda_0}} [V' (u^n_t (^{} 0))]$ is independent of
$n$ and $t$ because $u^n$ is stationary{\tmstrong{}}. Thus, we get for $p \in
\mathbb{N}$
\begin{align*}
   & \mathbb{E}_{{\mu}_{\lambda_0}} [| S^n_t (\eta) - S^n_s (\eta) |^{2 p}]
\\
   & \hspace{50pt} \leqslant | t - s |^{2 p - 1} \int_s^t \mathbb{E}_{{\mu}_{\lambda_0}}
   \left[ \left| \frac{1}{2} \sum_k n^{- 1 / 2} \Delta_n \eta (n^{- 1} k + c_n
   r) (V' (u^n_r (k)) - \varphi_{V'} (\rho' (\lambda_0))) \right|^{2 p}
   \right] \mathd r \\
   & \hspace{50pt} \lesssim | t - s |^{2 p - 1} \int_s^t \mathbb{E}_{{\mu}_{\lambda_0}}
   \left[ \left( \sum_k n^{- 1} | \Delta_n \eta (n^{- 1} k + c_n r) |^2 | V'
   (u^n_r (k)) - \varphi_{V'} (\rho' (\lambda_0)) |^2 \right)^p \right] \mathd
   r, 
\end{align*}
where we used the Burkholder-Davis-Gundy inequality for discrete time
martingales in the second step. Further, Minkowski's inequality yields
\begin{align*}
 &\mathbb{E}_{{\mu}_{\lambda_0}} \left[ \left( \sum_k n^{- 1} | \Delta_n
 \eta (n^{- 1} k + c_n r) |^2 | V' (u (k)) - \varphi_{V'} (\rho'
 (\lambda_0)) |^2 \right)^p \right] \\
 &\qquad \leqslant \left( \sum_k n^{- 1} | \Delta_n \eta (n^{- 1} k + c_n r) |^2
 \right)^p E_{\lambda_0} [(V' (u (0)) - \varphi_{V'} (\rho' (\lambda_0)))^{2
 p}],
\end{align*}
and as for the martingale contribution we get
\[ \sum_k n^{- 1} | \Delta_n \eta (n^{- 1} k + c_n r) |^2 \lesssim \| \Delta
   \eta \|_{H^1}^2, \]
so that overall
\[ \mathbb{E}_{{\mu}_{\lambda_0}} [| S^n_t (\eta) - S^n_s (\eta) |^{2 p}]
   \lesssim | t - s |^{2 p} \mathbb{E}_{{\mu}_{\lambda_0}} [(V' (u (0)) -
   \varphi_{V'} (\rho' (\lambda_0)))^{2 p}] \| \eta \|_{H^3}^{2 p} . \]
   As before, it now suffices to apply Mitoma's criterion and the Kolmogorov-Chentsov theorem to obtain the tightness of $(S^n)$ in $C ([0, 1], \CS')$.

\subsubsection{The antisymmetric term}\label{sec:antisymmetric}

It remains to control
\[ \mathd A^n_t (\eta) = \left( n^{3 / 2} \sum_k \eta (n^{- 1} k + c_n t)
   \CL_A u^n_t (k) + \langle v^n (t), n^{- 1} c_n \nabla \eta \rangle \right)
   \mathd t, \]
where
$ \CL_A = \sum_j n^{- 1 / 2} \nabla^{(2)}_D [V' (x)] (j) \partial_j$, 
and therefore $\CL_A u^n_t (k) = n^{- 1 / 2} \nabla^{(2)}_D [V' (u^n_t)] (k)$
and then
\[ \mathd A^n_t (\eta) = \left( - \sum_k \nabla^{(2)}_n \eta (n^{- 1} k + c_n
   t) V' (u^n_t (k)) + \langle v^n (t), c_n \nabla \eta \rangle \right) \mathd
   t. \]
We start by dealing with the second term on the right hand side, for which we
would like to replace the continuous gradient $\nabla \eta$ by $\nabla^{(2)}_n
\eta$ in order to cancel the diverging transport contribution in the first term. We have
\[ \int_s^t \langle v^n (r), c_n (\nabla \eta - \nabla^{(2)}_n \eta) \rangle
   \mathd r = \int_s^t c_n n^{- 1 / 2} \sum_k (u^n_r (k) - \rho' (\lambda_0))
   (\nabla \eta - \nabla^{(2)}_n \eta) (n^{- 1} k + c_n r) \mathd r \]
and
\begin{equation}
 \label{eq:transport correction vanishes} 
   \begin{aligned}
 &\mathbb{E}_{{\mu}_{\lambda_0}} \left[ \left| \int_s^t c_n n^{- 1 / 2}
 \sum_k (u^n_r (k) - \rho' (\lambda_0)) (\nabla \eta - \nabla^{(2)}_n \eta)
 (n^{- 1} k + c_n r) \mathd r \right|^2 \right]  \\
 &\quad \leqslant | t - s | \int_s^t n^{- 1} \sum_k c^2_n \sigma^2_{\lambda_0} |
 (\nabla \eta - \nabla^{(2)}_n \eta) (n^{- 1} k + c_n r) |^2 \mathd r  \\
 &\quad \lesssim | t - s | \int_s^t c^2_n \sigma^2_{\lambda_0} n^{- 2} \| \nabla
 \eta \|_{H^1}^2 \mathd r  
 \lesssim | t - s |^2 n^{- 1} \| \nabla \eta \|_{H^1}^2, 
   \end{aligned}
   \end{equation}
where in the last step we used that $c_n = n^{1 / 2} \partial_{\rho}
\varphi_{V'} (\rho' (\lambda_0))$. It remains to bound
\begin{align*}
 &\mathbb{E}_{{\mu}_{\lambda_0}} \left[ \left( \int_s^t \sum_k
 (\nabla^{(2)}_n \eta) (n^{- 1} k + c_n r) \big( V' \left( u^n_r \left( k
 \right) \right) - n^{- 1 / 2} c_n (u^n_r (k) - \rho' (\lambda_0)) \big)
 \mathd r \right)^2 \right] \\
 &\quad =\mathbb{E}_{{\mu}_{\lambda_0}} \left[ \left( \int_s^t \sum_k
 (\nabla^{(2)}_n \eta) (n^{- 1} k + c_n r) \left( V' (u^n_r (k)) -
 \varphi_{V'} \left( \rho' \left( \lambda_0 \right) \right) -
 \partial_{\rho} \varphi_{V'} (\rho' (\lambda_0)) (u^n_r (k) - \rho'
 (\lambda_0)) \right) \mathd r \right)^2 \right] \\
 &\quad \lesssim \left( \frac{\ell}{n^2} + \frac{| t - s |}{\ell} \right) \| (r, k)
 \mapsto (\nabla^{(2)}_n \eta) (n^{- 1} k + c_n r) \|_{L^2 ([0, T] \times
 \mathbb{Z})}^2 \sup_{\lambda} \tmop{var}_{\lambda} (V'),
\end{align*}
where in the last step we used the first order Boltzmann-Gibbs principle,
Theorem~\ref{thm:first order BG}, and $\ell \geqslant 1$ can be chosen
arbitrarily. Now
\[ \| (r, k) \mapsto (\nabla^{(2)}_n \eta) (n^{- 1} k + c_n r) \|_{L^2 ([0, T]
   \times \mathbb{Z})}^2 \lesssim | t - s | n \| \nabla \eta \|_{H^1}^2, \]
and since $V'$ is Lipschitz continuous we get $\tmop{var}_{\lambda} (V')
\lesssim \sigma_{\lambda}^2$. If $n^{- 2} \leqslant | t - s |$, we take $\ell
= \lfloor n | t - s |^{1 / 2} \rfloor$ and obtain
\begin{align*}
   & \left( \frac{\ell}{n^2} + \frac{| t - s |}{\ell} \right) \| (r, k) \mapsto (\nabla^{(2)}_n \eta) (n^{- 1} k + c_n r) \|_{L^2 ([0, T] \times \mathbb{Z})}^2 \sup_{\lambda} \tmop{var}_{\lambda} (V') \\
   & \hspace{150pt}\lesssim | t - s |^{3 / 2} \| \nabla \eta \|_{H^1} \sup_{\lambda}
   \sigma_{\lambda}^2 \lesssim | t - s |^{3 / 2} \| \nabla \eta \|_{H^1} .
\end{align*}
On the other side if $| t - s | < n^{- 2}$, we simply estimate
\begin{align*}
 &\mathbb{E}_{{\mu}_{\lambda_0}} \left[ \left( \int_s^t \sum_k
 (\nabla^{(2)}_n \eta) (n^{- 1} k + c_n r) 2 \left( V' (u^n_r (k)) -
 \varphi_{V'} \left( \rho' (\lambda_0) \right) - \partial_{\rho}
 \varphi_{V'} (\rho' (\lambda_0)) (u^n_r (k) - \rho' (\lambda_0)) \right)
 \mathd r \right)^2 \right] \\
 &\quad \leqslant (t - s) \int_s^t \sum_k | (\nabla^{(2)}_n \eta) (n^{- 1} k + c_n
 r) |^2 \\
 &\qquad \times \mathbb{E}_{{\mu}_{\lambda_0}} \left[ \left( \left( V' (u^n_r
 (k)) - \varphi_{V'} (\rho' (\lambda_0)) - \partial_{\rho} \varphi_{V'}
 (\rho' (\lambda_0)) (u^n_r (k) - \rho' (\lambda_0)) \right) \right)^2
 \right] \mathd r \\
 &\quad \lesssim \| \nabla \eta \|_{H^1}^2 | t - s |^2 n \lesssim \| \nabla \eta
 \|_{H^1}^2 | t - s |^{3 / 2}.
\end{align*}
In conclusion, we have shown that
\begin{equation}\label{eq:zero-qv}
   \mathbb{E}_{{\mu}_{\lambda_0}} [| A^n_t (\eta) - A^n_s (\eta) |^2]
   \lesssim | t - s |^{3 / 2} \| \nabla \eta \|_{H^1}^2 \lesssim | t - s |^{3
   / 2} \| \eta \|_{H^2}^2,
\end{equation}
   from where we get the tightness of $(A^n)$ in $C ([0, 1], \CS')$ by applying once more Mitoma's criterion and the Kolmogorov-Chentsov theorem.

Since all our tightness results are based on moment bounds, we also get the
joint tightness that we stated in Lemma~\ref{lem:tightness}.

\subsection{Identification of the limit}\label{sec:limit}

Here we show that every limit point of the sequence $(v^n)$ is an energy
solution of the stochastic Burgers equation. Since the uniqueness in law
of energy solutions was established in~\cite{bib:energyUniqueness}, this concludes the proof of
convergence. Let us first give the definition of energy solutions that was
introduced in \cite{bib:gubinelliJara}. Before we state it recall that the
quadratic variation in the sense of Russo and Vallois \cite{bib:russoVallois}
of a stochastic process $X$ with trajectories in $C ([0, 1],
\mathbb{R})$ is defined as
\[ [X]_t = \lim_{\delta \rightarrow 0} \int_0^t \frac{1}{\delta}
   (X_{s + \delta} - X_s)^2 \mathd s, \]
where the convergence is uniformly on compacts in probability (we write
\tmtextit{ucp}-convergence). If $X$ is a continuous semimartingale, then $[X]$
is nothing but its semimartingale quadratic variation.

\begin{definition}[Controlled process]  
  Denote with $\mathcal{Q}$ the space of pairs $(u, \mathcal{A})_{0 \leqslant
  t \leqslant T}$ of generalized stochastic processes with paths in $C \left(
  [0, 1], \CS' \right)$ such that
  \begin{enumerateroman}
    \item for all $t \in [0, 1]$ the law of $u_t$ is that of a spatial
    white noise with variance $\sigma_{\lambda_0}^2$;
    
    \item for any test function $\eta \in \CS$, the process $t \mapsto
    \mathcal{A}_t (\eta)$ is almost surely of zero quadratic variation in the
    sense of Russo and Vallois, $\mathcal{A}_0 (\eta) = 0$, and the pair $(u
    (\eta), \mathcal{A} (\eta))$ satisfies the equation
    \[ u_t (\eta) = u_0 (\eta) + \frac{\partial_{\rho} \varphi_{V'} (\rho'
       (\lambda_0))}{2} \int_0^t u_s (\Delta \eta) \tmop{ds} +\mathcal{A}_t
       (\eta) + M_t (\eta), \]
    where $M (\eta)$ is a martingale with respect to the filtration generated
    by $(u, \mathcal{A})$ with quadratic variation $[M (\eta)]_t = t \| \nabla
    \eta \|_{L^2 (\mathbb{R})}^2$;
    
    \item for $T > 0$ the reversed processes $\hat{u}_t = u_{T - t}$,
    $\hat{\mathcal{A}}_t = - (\mathcal{A}_T -\mathcal{A}_{T - t})$ satisfies
    for all $\eta \in \CS$
    \[ \hat{u}_t (\eta) = \hat{u}_0 (\eta) + \frac{\partial_{\rho}
       \varphi_{V'} (\rho' (\lambda_0))}{2} \int_0^t \hat{u}_s (\Delta \eta)
       \tmop{ds} + \hat{\mathcal{A}}_t (\eta) + \hat{M}_t (\eta), \]
    where $\hat{M} (\eta)$ is a martingale with respect to the filtration
    generated by $(\hat{u}, \hat{\mathcal{A}})$ with quadratic variation
    $[\hat{M} (\eta)]_t = t \| \nabla \eta \|_{L^2 (\mathbb{R})}^2$.
  \end{enumerateroman}
\end{definition}

\begin{remark}
  Recall from Lemma~\ref{lem:constants} that $\partial_{\rho} \varphi_{V'} (\rho) = \sigma_{h'(\rho)}^{-2}$, which shows that the factor in front of the Laplacian is always strictly positive.
\end{remark}

We will need the following results that both will be included in the revised version of~\cite{bib:energyUniqueness}.

\begin{lemma}[see~{\cite{bib:energyUniqueness}}]\label{lem:Burgers nonlinearity}
  Let $(u, \mathcal{A}) \in \mathcal{Q}$, and $\chi \in L^1 (\mathbb{R}) \cap
  L^2 (\mathbb{R})$ with $\int_{\mathbb{R}} \chi (x) \mathd x = 1$. Define
  for $\delta > 0 \nocomma$, $x \in \mathbb{R}$ the rescaled function
  $\chi^{\delta}_x (\cdot) = \delta^{- 1} \chi (\delta (x -
  \cdummy)) .$ Then for all $\eta \in \CS$ the ucp-limit
  \[  \int_0^{\cdummy} \nabla u_s^2 (\eta) \mathd s \assign \lim_{\delta \rightarrow 0} \int_0^{\cdummy} \int_{\mathbb{R}}
     \nabla [u_s (\chi^{\delta}_x)^2] \eta (x) \mathd x \mathd s \]
  exists and does not depend on $\chi$.
\end{lemma}

\begin{definition}[Energy solution]
  A controlled process $(u, \mathcal{A}) \in \mathcal{Q}$ is an
  \tmtextbf{energy solution} to the stochastic Burgers equation if for all
  $\eta \in \CS$ almost surely
  \[ \mathcal{A} (\eta) = \frac{(\partial_{\rho \rho} \varphi_{V'}) (\rho'
     (\lambda_0))}{2} \int_{0^{}}^{\cdummy} \nabla u_s^2 (\eta) \tmop{ds} . \]
\end{definition}

\begin{theorem}[see~{\cite{bib:energyUniqueness}}]
   There is a unique energy solution to the stochastic Burgers equation.
\end{theorem}

In Lemma~\ref{lem:tightness} it was shown that
\[ \mathd v^n_t = \mathd S^n_t + \mathd A^n_t + \mathd M^n_t, \]
where $S^n$ is the symmetric part, $A^n$ the antisymmetric part, and $M^n$ a
martingale, with each one of them being tight in $C \left( [0, 1], \CS'
\right)$. If now $T > 0$ and $\hat{u}^n_t = u^n_{T - t}$ and
\[ \hat{v}^n_t = v^n_{T - t} = \sum_k n^{1 / 2} (u^n_{T - t} (k) - \rho'
   (\lambda_0)) n^{- 1} \delta_{n^{- 1} (k + c_n (T - t))} \]
then by Corollary~\ref{cor:reversedEvolution} we get
\[ \mathd \hat{v}^n_t = \mathd \hat{S}^n_t + \mathd \hat{A}^n_t + \mathd
   \hat{M}^n_t, \]
where $\hat{M}^n$ is a martingale with quadratic variation
\[ [\hat{M}^n (\eta)]_t = n^{- 2} \sum_k \int_0^t | \nabla^{(1)}_n \eta (n^{-
   1} k + c_n (T - r)) |^2 n \mathd r, \]
and $\hat{S}^n$ and $\hat{A}^n$ satisfy
\[ \mathd \hat{S}^n_t (\eta) = \sum_k n^{- 1 / 2} \eta (n^{- 1} k + c_n (T -
   t)) \CL^{(n)}_S \hat{u}^n_t (k) \mathd t \]
and
\[ \mathd \hat{A}^n_t (\eta) = \Big( \sum_k n^{- 1 / 2} \eta (n^{- 1} k + c_n
   (T - t)) (- \CL^{(n)}_A) \hat{u}^n_t (k) - \langle \hat{v}^n_t, n^{- 1} c_n
   \nabla \eta \rangle \Big) \mathd t, \]
respectively. Thus, the same arguments as in Section~\ref{sec:tightness} give
the joint tightness of
\[ (v^n_0, S^n, A^n, M^n, \hat{S}^n, \hat{A}^n, \hat{M}^n)_{n \in \mathbb{N}}
\]
in $\CS' \times C \left( [0, T], \CS' \right)^6$. From now on we fix a
converging subsequence, which by abuse of notation we index again by $n \in
\mathbb{N}$. Obviously the following theorem then implies
our main result, Theorem~\ref{thm:main result}.

\begin{theorem}
  For $v = \lim_n v^n$ and $\mathcal{A}= \lim_n A^n$ we have $(v, \mathcal{A})
  \in \mathcal{Q}$ and the pair is an energy solution to the stochastic
  Burgers equation.
\end{theorem}

\begin{proof}
  Let us write also $X = \lim_n X^n$ for $X = S, \hat{S}, M, \hat{M}$, and
  $\hat{\mathcal{A}} = \lim_n \hat{A}^n$. We have to show that
  \begin{enumerateroman}
    \item for all $t \in [0, T]$ the law of $v_t$ is that of a space white
    noise with variance $\sigma_{\lambda_0}^2$;
    
    \item $M$ is a martingale in the filtration of $(v, \mathcal{A})$, with
    quadratic variation $[M (\eta)]_t = t | | \nabla \eta | |_{L^2}$;
    
    \item $\mathcal{A}$ has zero quadratic variation;
    
    \item $S_t (\eta) = \frac{\partial_{\rho} \varphi_{V'} (\rho'
    (\lambda_0))}{2} \int_{0^{}}^t v_s (\Delta \eta) \tmop{ds}$;
    
    \item the same conditions hold for the reversed process and
    $\hat{\mathcal{A}} = - (\mathcal{A}_T -\mathcal{A}_{T - t})$;
    
    \item for every $\eta \in \CS$ we have $\mathcal{A}_t (\eta) =
    \frac{(\partial_{\rho \rho} \varphi_{V'}) (\rho' (\lambda_0))}{2} \int_0^t
    \nabla u_s^2 (\eta) \mathd s$.
  \end{enumerateroman}
  Point i. was already shown in Section~\ref{sec:fixedTimes}.  As for ii., let $\Phi : \mathbb{R}^{2 N} \rightarrow \mathbb{R}$ be a
  bounded measurable function, let $0 \leqslant t_1 < \ldots < t_{N + 1}
  \leqslant T$, and let $\eta_1, \ldots, \eta_{N + 1} \in \CS$. Then the
  uniform moment bounds for $M^n$ that we derived in
  Section~\ref{sec:martingale} show that
  \begin{align*}
     & \mathbb{E} [\Phi (v_{t_1} (\eta_1), \mathcal{A}_{t_1} (\eta_1), \ldots,
     v_{t_N} (\eta_N), \mathcal{A}_{t_N} (\eta_N)) (M_{t_{N + 1}} (\eta_{N +
     1}) - M_{t_N} (\eta_N))] \\
      &\qquad = \lim_{n \rightarrow \infty} \mathbb{E} [\Phi (v^n_{t_1} (\eta_1),
     \mathcal{A}^n_{t_1} (\eta_1), \ldots, v_{t_N} (\eta_N),
     \mathcal{A}^n_{t_N} (\eta_N)) (M^n_{t_{N + 1}} (\eta_{N + 1}) - M^n_{t_N}
     (\eta_N))] = 0.
  \end{align*}
     Now we obtain from the monotone class theorem (see e.g. \cite[Corollary~4.4]{bib:ethierKurtz}) that $M$ is a martingale in the
  filtration generated by $(v, \mathcal{A})$. The same argument also shows
  that for $\eta \in \CS$ the process $(M_t^2 (\eta) - t \| \nabla \eta
  \|_{L^2 (\mathbb{R})}^2)_{t \in [0, T]}$ is a martingale, and therefore $[M (\eta)]_t = t \| \nabla \eta \|^2_{L^2
  (\mathbb{R})}$. Moreover, the same line of reasoning also works for the
  backward martingale $\hat{M}$.
  
  Next let us show iii., that is $[\mathcal{A} (\eta)] \equiv 0$, where the
  quadratic variation is in the sense of Russo and Vallois. We showed in~\eqref{eq:zero-qv} that $\mathbb{E}_{{\mu}_{\lambda_0}} [|
  A^n_t (\eta) - A^n_s (\eta) |^2] \lesssim | t - s |^{3 / 2}$, uniformly in
  $n.$ \ By Fatou's lemma we then get
  \begin{eqnarray*}
    \mathbb{E}_{{\mu}_{\lambda_0}} \left[ \int_0^t \frac{1}{\delta}
    | \mathcal{A}_{s + \delta} (\eta) -\mathcal{A}_s (\eta) |^2 \tmop{ds}
    \right] & \leqslant & \liminf_n \mathbb{E}_{{\mu}_{\lambda_0}} \left[
    \int_0^t \frac{1}{\delta} | A^n_{s + \delta} (\eta) - A^n_s
    (\eta) |^2 \tmop{ds} \right]\\
    & \lesssim & \int_0^t \frac{1}{\delta} \delta^{3 / 2} \tmop{ds}
    \xrightarrow{\delta \rightarrow 0} 0.
  \end{eqnarray*}
  For iv. it suffies to show that for all $\eta \in \CS$ and $t \in
  [0, T]$ we have
  \[ \lim_{n \rightarrow \infty} \mathbb{E}_{{\mu}_{\lambda_0}} \Big[
     \Big( S^n_t (\eta) - \frac{\partial_{\rho} \varphi_{V'} (\rho'
     (\lambda_0))}{2} \int_0^t v^n_s (\triangle \eta) \tmop{ds} \Big)^2
     \Big] = 0. \]
  As we have seen in Section \ref{sec:symmetric}, $S_t^n (\eta) = \int_0^t
  \frac{1}{2} \sum_k n^{- 1 / 2} \Delta_n \eta (n^{- 1} k + c_n s) (V' (u^n_s
  (k)) - \varphi_{V'} (\rho' (\lambda_0))) \mathd s$, and therefore
  \begin{eqnarray*}
    \mathbb{E} \Big[ \Big( S^n_t (\eta) - \partial_{\rho} \varphi_{V'}
    (\rho' (\lambda_0)) \int_0^t v^n_s (\triangle \eta) \tmop{ds} \Big)^2
    \Big] & \lesssim & X + Y
  \end{eqnarray*}
  with
  \begin{eqnarray*}
    X & = & \mathbb{E} \Big[ \Big( \int_0^t \frac{1}{2} \sum_k n^{- 1 / 2}
    \Delta_n \eta (n^{- 1} k + c_n s) \times \\
    && \hspace{100pt} \times ( V' (u^n_s (k)) - \varphi_{V'} ( \rho' (\lambda_0) ) - \partial_{\rho} \varphi_{V'} (\rho' (\lambda_0)) (u_s^n (k) - \rho' (\lambda_0)) ) \mathd s^{} \Big)^2
    \Big]\\
    & \lesssim & \frac{1}{n^2} \| (r, k) \mapsto \triangle_n \eta (n^{- 1} k
    + c_n r) \|_{L^2 ([0, T] \times \mathbb{Z})}^2 \sup_{\lambda}
    \tmop{var}_{\lambda} (V')
  \end{eqnarray*}
  by the first order Boltzmann-Gibbs principle, Theorem \ref{thm:first order
  BG}, with $\ell = n$. Now we have
  \[ \| (r, k) \mapsto \triangle_n \eta (n^{- 1} k + c_n r) \|_{L^2 ([0, T]
     \times \mathbb{Z})}^2 \lesssim n \| \eta \|_{H^3}^2, \]
  so that $X \lesssim n^{- 1}$ goes to zero. The remaining contribution is
  \begin{align*}
     Y & =\mathbb{E} \Big[ \Big( \int_0^t \frac{1}{2} \sum_k n^{- 1 / 2} [\Delta_n \eta (n^{- 1} k + c_n s) - \triangle_{} \eta (n^{- 1} k + c_n s)] \times \\
     &\hspace{260pt} \times \partial_{\rho} \varphi_{V'} (\rho' (\lambda_0)) (u_s^n (k) - \rho' (\lambda_0))) \mathd s^{} \Big)^2 \Big] \\
     & \leqslant t \int_0^t n^{- 1} \sum_k \mathbb{E} \left[ \left( [\Delta_n
     \eta (n^{- 1} k + c_n s) - \triangle_{} \eta (n^{- 1} k + c_n s)] \left.
     \partial_{\rho} \varphi_{V'} (\rho' (\lambda_0)) (u_s^n (k) - \rho'
     (\lambda_0)) \right) \mathd s^{} \right)^2 \right] \\
     & \lesssim n^{- 3} \| \eta \|_{H^3}^2,
  \end{align*}
  which also vanishes for $n \rightarrow \infty$. The same arguments show that
  also
  \[ \hat{S}_t (\eta) = \frac{\partial_{\rho} \varphi_{V'} (\rho'
     (\lambda_0))}{2} \int_{0^{}}^t \hat{v}_s (\Delta \eta) \tmop{ds}. \]
   To see that
  $\hat{\mathcal{A}} = - (\mathcal{A}_T -\mathcal{A}_{T - t})$ it suffices to
  note that for fixed $n \in \mathbb{N}$ we have $\hat{A}^n_t (\eta) = - (A^n_T (\eta) - A^n_{T - t} (\eta))$ by definition, and the equality carries over to the limit $n\to \infty$. 
  Thus, v. is established as well and in particular $(v, \mathcal{A}) \in
  \mathcal{Q}$.
  
  It remains to show point vi. about $\mathcal{A}$ being equal to the Burgers
  nonlinearity. For $\delta > 0$ define
  \[ \mathcal{A}^{\delta}_t (\eta) \assign  \frac{(\partial_{\rho \rho} \varphi_{V'}) (\rho' (\lambda_0))}{2} \int_0^t \mathd r \int_{\mathbb{R}}
     \mathd x v_s (i_{\delta} (x))^2 \nabla \eta (x), \]
  where $i_{\delta} (x) (y) = \delta^{- 1} 1_{[x, x + \delta)}
  (y)$ which by Lemma~\ref{lem:Burgers nonlinearity} is allowed as a test
  function. The proof is complete once we show that $\mathcal{A}^{\delta}
  (\eta)$ converges in probability to $\mathcal{A} (\eta)$ as $\delta
  \rightarrow 0$. Now we have seen in Section~\ref{sec:antisymmetric} that
  \[ A^n_t (\eta) = \int_0^t \sum_k (\nabla^{(2)}_n \eta) (n^{- 1} k + c_n r)
     (V' (u^n_r (k)) + \partial_{\rho} \varphi_{V'} (\rho' (\lambda_0)) (\rho'
     (\lambda_0)) (u^n_r (k) - \rho' (\lambda_0))) \mathd r + O (n^{- 1}), \]
  so that the second order Boltzmann-Gibbs principle, Theorem~\ref{thm:second
  order BG}, applied with $\ell = \delta n$ gives
  \begin{align*}
     & \lim_{n \rightarrow \infty} \mathbb{E}_{{\mu}_{\lambda_0}} \Big[
     \Big( A^n_t (\eta) - \int_0^t \sum_k (\nabla^{(2)}_n \eta) (n^{- 1} k +
     c_n r) \tau_k \frac{(\partial_{\rho \rho} \varphi_{V'}) (\rho'
     (\lambda_0))}{2} \CQ_{\lambda_0} (\delta n ; u^n_r) \Big)^2
     \Big] \\
     &\qquad \lesssim_{\eta} \lim_{n \rightarrow \infty} \Big( \frac{\delta n}{n^2} + \frac{t}{\delta^2 n^2} \Big) t n \sup_{\lambda} \tmop{var}_{\lambda} (V') \lesssim \delta t.
  \end{align*}
  So since $(A^n_t (\eta))_n$ converges to $\mathcal{A}_t (\eta)$, it suffices to show that
%
%
  \[ \Big( \frac{(\partial_{\rho \rho} \varphi_{V'}) (\rho' (\lambda_0))}{2}
     \Big)^{- 1} \mathcal{A}^{\delta}_t = \lim_{n \rightarrow \infty}
     \int_0^t \mathd r \sum_k \tau_k \mathcal{Q}_{\lambda_0} (\delta n ;
     u^n_r) (\nabla^{(2)}_n \varphi) (n^{- 1} k + c_n r) \]
  to complete the proof. Now recall that $\CQ_{\lambda_0} (\ell ; u) =
  (\bar{u}^{\ell} - \rho' (\lambda_0))^2 - \frac{\sigma^2_{\lambda_0}}{\ell}$
  with $\bar{u}^{\ell} = \ell^{- 1} \sum_{i = 0}^{\ell - 1} u (i)$ and that $v^n_t = \sum_k n^{- 1 / 2} (u^n_t (k) - \rho' (\lambda_0)) \delta_{n^{-
     1} (k + c_n t)}$.
  Hence,
  \[ n \tau_k \mathcal{Q}_{\lambda_0} (\delta n ; u^n_t) = v^n_t
     (i_{\delta} (n^{- 1} k + c_n t))^2 -
     \frac{\sigma^2_{\lambda_0}}{\delta}, \]
  and then
  \begin{align}
     &\int_0^t  \sum_k \tau_k \mathcal{Q} (\delta n ; u^n_r) (\nabla^{(2)}_n \varphi) (n^{- 1} k + c_n r) \mathd r\notag \\
     &\qquad = \int_0^t \sum_k n^{- 1} v^n_r (i_{\delta} (n^{- 1} k + c_n r))^2 (\nabla^{(2)}_n \varphi) (n^{- 1} k + c_n r) \mathd r\notag \\
     &\qquad \rightarrow \int_0^t \mathd r \int_{\mathbb{R}} \mathd x v_s (i_{\delta} (x))^2 \nabla \varphi (x), \label{eq:approxNonlin} 
  \end{align}
  where the last convergence is in $L^2$ and can be proven as follows. For
  simplicity we only treat the case $\delta = 1$, $x = 0$. Let $i_{}^m
  (0)$ be a smooth approximation to $i (0) \assign i_1 (0)$, to be specified below. Since $v_s$ is a multiple of the white noise we have
  \[ \mathbb{E}_{{\mu}_{\lambda_0}} [\{ v_s (i_{} (0)) - v_s (i_{}^m (0))
     \}^2] \lesssim \| i_{} (0) - i_{}^m (0) \|_{L^2 (\mathbb{R})}^2 . \]
  Assume now that $i^m (0) (y) = i (0) (y) = 1$ on $[0, 1]$, and that $i^m
  (0)$ is zero outside $[- 1 / m, 1 + 1 / m]$ and $i^m (0) (y) \in [0, 1]$ for
  all $y \in \mathbb{R}$. Then
     \begin{align*}
   \mathbb{E}_{{\mu}_{\lambda_0}} [\{ v^n_s (i_{} (0)) - v^n_s (i_{}^m
   (0)) \}^2]
   &= \mathbb{E}_{{\mu}_{\lambda_0}} \Big[ \Big( \sum_k
   \frac{1}{\sqrt{n}} [u (n^2 t, k) - \rho' (\lambda_0)] [i (0) - i^m (0)]
   (n^{- 1} \{ k + c_n t \}) \Big)^2 \Big] \\
   &= \sum_k \frac{1}{n} \mathbb{E}_{{\mu}_{\lambda_0}} [u (n^2 t, k) -
     \rho' (\lambda_0)]^2 [i (0) - i^m (0)]^2 (n^{- 1} \{ k + c_n t \})
     \lesssim \frac{1}{m},
     \end{align*}
  since the support of $\{ i (0) - i^m (0) \} (n^{- 1} (\cdot + c_n t))$ has
  size of order $n / m$. These two bounds give the convergence
  in~(\ref{eq:approxNonlin}).
\end{proof}

\appendix\section{Infinite-dimensional SDEs}\label{app:periodic approximation}

Define for $r > 0$
\[ \tilde{\chi}_r \assign \Big\{ u \in \mathbb{R}^{\mathbb{Z}} : \| u \|_r
   \assign \Big( \sum_j | u (j) |^2 | j |^{- r} \Big)^{1 / 2} < \infty
   \Big\}, \]
where we simplify the notation by setting $| 0 |^{- r} = 1$, and then $\chi_r$
as the closure of all $u \in \tilde{\chi}_r$ that only have finitely many
non-zero entries. Clearly $\tilde{\chi}_r$ is a Banach space, and therefore
$\chi_r$ as well, and by definition $\chi_r$ is also separable. We have an
intrinsic description of $\chi_r$ as
\[ \chi_r = \Big\{ u \in \tilde{\chi}_r : \lim_{N \rightarrow \infty} \sum_{|
   j | > N} | u (j) |^2 | j |^{- r} = 0 \Big\} . \]
We then define
\[ \mathcal{U}_r = \Big\{ u : [0, T] \times \Omega \rightarrow \chi_r,
   \hspace{10pt} u \text{ is almost surely continuous and } \mathbb{E} [\sup_{t \leqslant T} \| u_t \|_r^2] < \infty
   \Big\}, \]
which is also a Banach spaces if we identify processes that are
indistinguishable.

\begin{lemma}\label{lem:picard}
  Let $r \in \mathbb{R}$, let $(v_t)_{t \in [0, T]}$ be a
  stochastic process with continuous trajectories in
  $\mathbb{R}^{\mathbb{Z}}$ (equipped with the product topology), and assume that
  $\mathbb{E} [ \int_0^T \| v_t \|_r^2 \mathd r ] < \infty$. Let
  $\gamma \in \mathcal{U}_r$ and let $u_0$ be a random variable with values in
  $\mathbb{R}^{\mathbb{Z}}$ such that $\mathbb{E} [\| u_0 \|_r^2] <
  \infty$. Let $V'$ be Lipschitz-continuous and define
  \[ \Gamma (v)_t (j) \assign u_0 (j) + \int_0^t \Big( \frac{1}{2} \Delta_D
     [V' (v_t)] (j) + 2 \sqrt{\varepsilon} \nabla^{(2)}_D [V' (v_t)] (j)
     \Big) \mathd t + \gamma_t (j) . \]
  Then
  \[ \mathbb{E} [\sup_{t \leqslant T} \| \Gamma (v)_t \|_r^2] \lesssim
     \mathbb{E} [\| u_0 \|_r^2] + T\mathbb{E} \Big[ \int_0^T \| v_t \|_r^2
     \mathd r \Big] +\mathbb{E} [\sup_{t \leqslant T} \| \gamma_t \|_r^2] .
  \]
\end{lemma}

\begin{proof}
  Simply estimate
  \begin{align*}
   \mathbb{E} [\sup_{t \leqslant T} | \Gamma (v)_t (j) |^2]
   &\lesssim
   \mathbb{E} [| u_0 (j) |^2] + T \int_0^T \mathbb{E} \Big[ \Big|
   \frac{1}{2} \Delta_D [V' (v_t)] (j) + 2 \sqrt{\varepsilon} \nabla^{(2)}_D
   [V' (v_t)] (j) \Big|^2 \Big] \mathd t \\
   & \quad +\mathbb{E} [\sup_{t \leqslant T} | \gamma_t (j) |^2] \\
   &\lesssim \mathbb{E} [| u_0 (j) |^2] + T \int_0^T \sum_{| k - j |
   \leqslant 1} \mathbb{E} [| v_t (k) |^2] \mathd t +\mathbb{E} [\sup_{t
   \leqslant T} | \gamma_t (j) |^2],
  \end{align*}
  and therefore
  \[ \mathbb{E} [\sup_{t \in \leqslant T} \| \Gamma (v)_t \|_r^2] \lesssim
     \mathbb{E} [\| u_0 \|_r^2] + T\mathbb{E} \Big[ \int_0^T \| v_t \|_r^2
     \mathd t \Big] +\mathbb{E} [\sup_{t \leqslant T} \| \gamma_t \|_r^2] .
  \]
  
\end{proof}

\begin{theorem}\label{thm:periodic approximation}
   Let $\gamma \in \mathcal{U}_r$, let
  $u_0$ be a random variable with values in $\chi_r$ such that $\mathbb{E}
  [\| u_0 \|_r^2] < \infty$, and let $V'$ be Lipschitz-continuous. Then there exists a unique solution $u
  \in \mathcal{U}_{r}$ to the equation
  \[ u_t (j) = u_0 (j) + \int_0^t \Big( \frac{1}{2} \Delta_D [V' (u_t)] (j) +
     \sqrt{\varepsilon} \nabla^{(2)}_D [V' (u_t)] (j) \Big) \mathd t +
     \gamma_t (j), \hspace{2em} j \in \mathbb{Z}. \]
  Moreover, if $\tilde{\gamma} \in \mathcal{U}_r$ and $\tilde{u}_0 \in \chi_r$
  with $\mathbb{E} [\| \tilde{u}_0 \|_r^2] < \infty$ is another set of data
  and $\tilde{u}$ denotes the corresponding solution, then
  \[ \mathbb{E} [\sup_{t \leqslant T} \| u_t - \tilde{u}_t \|_r^2] \lesssim
     \mathbb{E} [\| u_0 - \tilde{u}_0 \|_r^2] +\mathbb{E} [\sup_{t \leqslant
     T} \| \gamma_t - \tilde{\gamma}_t \|_r^2] . \]
  In particular, if $u_0^N$ is the periodic extension of $u_0 |_{[- N / 2, N /
  2)}$ to $\mathbb{Z}$ and similarly for $\gamma^N$, and $u^N$ is the
  corresponding solution, then
  \[ \lim_{N \rightarrow \infty} \mathbb{E} [\sup_{t \leqslant T} \| u_t -
     u^N_t \|_{r'}^2] = 0 \]
  for all $r' > r$.
\end{theorem}

\begin{proof}
  The existence and uniqueness of the solution $u \in \mathcal{U}_r$ follows
  from a Picard iteration and Gronwall's lemma. If $u$ and $\tilde{u}$ are the
  solutions for different data, then we get as in Lemma~\ref{lem:picard}
  \[ \mathbb{E} [\sup_{t \leqslant T} \| u_t - \tilde{u}_t \|_r^2] \lesssim
     \mathbb{E} [\| u_0 - \tilde{u}_0 \|_r^2] + T\mathbb{E} \Big[ \int_0^T
     \| u_t - \tilde{u}_t \|_r^2 \mathd r \Big] +\mathbb{E} [\sup_{t
     \leqslant T} \| \gamma_t - \tilde{\gamma}_t \|_r^2], \]
  and thus the claim follows from Gronwall's lemma.
\end{proof}

\begin{remark}
  In our setting we have $\gamma (j) = W (j + 1) - W (j)$ for an independent. family
  of standard Brownian motions $(W (j))_{j \in \mathbb{Z}}$, and therefore
  \[ \mathbb{E} [\sup_{t \leqslant T} \| \gamma_t \|_r^2] \leqslant \sum_j
     \mathbb{E} [\sup_{t \leqslant T} | W_t (j + 1) - W_t (j) |^2] | j |^{-
     r} < \infty \]
  for all $r > 1$. Similarly, since the coordinates under the measure ${\mu}_{\lambda}$ have
  finite second moments,
  \[ \mathbb{E} [\| u_0 \|_r^2] = \sum_j \mathbb{E} [| u_0 (j) |^2] | j |^{-
     r} < \infty \]
  for $r > 1$. Therefore, our solution takes values in $\chi_r$ whenever $r > 1$.
\end{remark}

\begin{remark}
  Our polynomial choice of weights is different from the exponential choice of
  \cite{bib:funakiSpohn, bib:giacominOllaSpohn}. The reason why
  we prefer it is that the polynomial weights give us a more accurate
  description of the growth of the solution (which will be at most
  logarithmic because under the stationary measure the coordinates are
  identically distributed with finite exponential moments), and more importantly we know
  that if $u \in \chi_r$, then
  \[ \sum_j u (j) \delta_j \]
  defines a tempered distribution and not just a distribution, which is not
  obvious when working with exponential weights.
\end{remark}

\section{Proof of the equivalence of ensembles}\label{app:equiv-ensem-proof}

Here we provide the proof of Proposition~\ref{prop:equiv-ensem}. Throughout this appendix we will not deal
with any temporal dependencies and therefore we will denote spatial
coordinates by subscripts, writing $u_k$ instead of $u (k)$. Let us start by deriving some auxiliary lemmas.

\begin{lemma}\label{lem:Gaussian expansion}
  With the notation of Lemma \ref{lem:LLT}, we can estimate uniformly in $\lambda \in
  \mathbb{R}$
  \[ \left| r_{\lambda, N} (u) - r_{\lambda, N} (0) + \frac{N^{- 1 / 2}}{2
     \sqrt{2 \pi}} \frac{m_{3, \lambda}}{\sigma^3_{\lambda}} x + \frac{x^2}{2
     \sqrt{2 \pi}} \right| \lesssim | u |^3 + N^{- 1 / 2} | u |^2 + N^{- 1} |
     u | . \]
\end{lemma}

\begin{proof}
  We perform a Taylor expansion. Clearly
  \[
      \left| r^0 (u) - r^0 (0) + \frac{1}{2 \sqrt{2 \pi}} u^2 \right| = \left| r^0 (x) - r^0 (0) - (r^0)' (0) u - \frac{1}{2} (r^0)'' (0) u^2
     \right| 
     \lesssim | u |^3 \]
  and
  \[
     N^{- 1 / 2} \left| r^1_{\lambda} (u) - r^1_{\lambda} (0) +
     \frac{1}{2 \sqrt{2 \pi}} \frac{m_{3, \lambda}}{\sigma^3_{\lambda}} u
     \right| = N^{- 1 / 2} | r^1_{\lambda} (u) - r^1_{\lambda} (0) - (r^1_{\lambda})'
     (0) u |  \lesssim N^{- 1 / 2} \| (r^1_{\lambda})'' \|_{\infty} u^2, \]
     and it easily follows from Lemma \ref{lem:UB} that $\|
  (r^1_{\lambda})'' \|_{\infty}$ is uniformly bounded in $\lambda$. Finally,
  \[ N^{- 1} | (r^2_{\lambda} (u) - r^2_{\lambda} (0)) | \leqslant N^{- 1} \|
     (r^2_{\lambda})' \|_{\infty} | u |, \]
     and again by Lemma \ref{lem:UB} the term $\|
  (r^2_{\lambda})' \|_{\infty}$ is uniformly bounded in $\lambda$.
\end{proof}

\begin{lemma}
  \label{lem:equiv ensem}Let $\ell \leqslant N / 2$ and let $F \in L^2
  ({\mu}_{h' (\rho)})$ depend only on $u_0, \ldots, u_{\ell - 1}$. Write
  \[ \psi_F (N, \rho) = E_{\lambda_0} [F| \bar{u}^N = \rho], \hspace{2em}
     \varphi_F (\rho) = E_{h' (\rho)} [F] \]
  (recall that $\bar{u}^N = N^{- 1} u^N = N^{- 1} \sum_{k = 0}^{N - 1} u_k$).
  Then
  \begin{align}\label{eq:equiv ensem lhs} \nonumber
     & \left| \psi_F (N, \rho) - \left( 1 +
    \frac{\ell}{2 N} \right) \varphi_F (\rho) - \frac{1}{2 N} \frac{m_{3, h'
    (\rho)}}{\sigma^4_{h' (\rho)}} E_{h' (\rho)} [F (u^{\ell} - \ell \rho)] +
    \frac{1}{2 N \sigma_{h' (\rho)}^2} E_{h' (\rho)} [F (u^{\ell} - \ell
    \rho)^2] \right| \\
    & \hspace{50pt} \lesssim \left( \frac{\ell}{N} \right)^{3 / 2} (\tmop{var}_{h' (\rho)} (F))^{1 / 2}.
  \end{align}
\end{lemma}

\begin{proof}
  Note that neither the left hand side nor the right hand side of~(\ref{eq:equiv ensem lhs}) change
  if we add a constant to $F$, so we can suppose without loss of generality
  that $E_{h' (\rho)} [F] = 0$. By~(\ref{eq:conditional density}) above we
  have
  \[ \psi_F (N, \rho) = \int F (u_0, \ldots, u_{\ell - 1}) \frac{p_{h' (\rho)}
     (u_0) \ldots p_{h' (\rho)} (u_{\ell - 1}) p^{\ast (N - \ell)}_{h' (\rho)}
     (N \rho - u^{\ell})}{p_{h' (\rho)}^{\ast N} (N \rho)} \mathd u_0 \ldots
     \mathd u_{\ell - 1}, \]
  and therefore with $c_1 = \frac{1}{2 N} \frac{m_{3, h' (\rho)}}{\sigma^4_{h'
  (\rho)}}$ and $c_2 = - \frac{1}{2 N \sigma_{h' (\rho)}^2}$
  \begin{align*}
     &\psi_F (N, \rho) - \left( 1 + \frac{\ell}{2 N} \right) \varphi_F (\rho) -
     c_1 E_{h' (\rho)} [F (u^{\ell} - \ell \rho)] - c_2 E_{h' (\rho)} [F
     (u^{\ell} - \ell \rho)^2] \\
     & \hspace{50pt} = \int \mathd u_0 \ldots \mathd u_{\ell - 1} F (u_0, \ldots, u_{\ell -
     1}) p_{h' (\rho)} (u_0) \ldots p_{h' (\rho)} (u_{\ell - 1}) \\
     & \hspace{70pt} \times \left( \frac{p^{\ast (N - \ell)}_{h' (\rho)} (N \rho - u^{\ell})}{p_{h' (\rho)}^{\ast N} (N \rho)} - \left( 1 + \frac{\ell}{2 N} \right) - c_1 (u^{\ell} - \ell \rho) - c_2 (u^{\ell} - \ell \rho)^2 \right).
  \end{align*}
  Since $p^{\ast
  (N - \ell)}_{h' (\rho)}$ is the density of $\sum_{k = \ell}^{N - 1} U_k
  \nocomma$, we get
  \[ p^{\ast (N - \ell)}_{h' (\rho)} (N \rho - u^{\ell}) = (N - \ell)^{- 1 /
     2} \sigma_{h' (\rho)}^{- 1} f^{N - \ell}_{h' (\rho)} (y), \]
  for $y = (\ell \rho - u^{\ell}) / \left( \sigma_{h' (\rho)} \sqrt{N - \ell}
  \right)$, and similarly
  \[ p_{h' (\rho)}^{\ast N} (N \rho) = N^{- 1 / 2} \sigma_{h' (\rho)}^{- 1}
     f^N_{h' (\rho)} (0) . \]
  Thus we need to bound
  \begin{align*}
     & \left| \frac{(N - \ell)^{- 1 / 2} f^{N - \ell}_{h' (\rho)} (y)}{N^{- 1 /  2} f^N_{h' (\rho)} (0)} - \left( 1 + \frac{\ell}{2 N} \right) - c_1  (u^{\ell} - \ell \rho) - c_2 (u^{\ell} - \ell \rho)^2 \right| \\
     & \hspace{20pt} = \frac{1}{f^N_{h' (\rho)} (0)} \left| \left( 1 - \frac{\ell}{N} \right)^{- 1 / 2} f^{N - \ell}_{h' (\rho)} (y) - \left( \left( 1 +
     \frac{\ell}{2 N} \right) + c_1 (u^{\ell} - \ell \rho) + c_2 (u^{\ell} -  \ell \rho)^2 \right) f^N_{h' (\rho)} (0) \right| .
  \end{align*}
  The first factor can be simply estimated by $f^N_{h' (\rho)} (0)^{- 1}
  \lesssim 1$, and using Lemma \ref{lem:LLT} we get
  \begin{align*}
     &\left| \left( 1 - \frac{\ell}{N} \right)^{- 1 / 2} f^{N - \ell}_{h'
     (\rho)} (y) - \left( \left( 1 + \frac{\ell}{2 N} \right) + c_1 (u^{\ell}
     - \ell \rho) + c_2 (u^{\ell} - \ell \rho)^2 \right) f^N_{h' (\rho)} (0)
     \right| \\
     & \hspace{50pt} \lesssim (1 + c_1 | u^{\ell} - \ell \rho | + c_2 | u^{\ell} - \ell \rho
     |^2) N^{- 3 / 2} \\
     & \hspace{70pt} + \left| \left( 1 - \frac{\ell}{N} \right)^{- 1 / 2} r_{h' (\rho), N} (y)
     - \left( \left( 1 + \frac{\ell}{2 N} \right) + c_1 (u^{\ell} - \ell \rho)
     + c_2 (u^{\ell} - \ell \rho)^2 \right) r_{h' (\rho), N} (0) \right| . 
   \end{align*}
  The second term on the right hand side is bounded by
  \begin{align*}
   &\left| \left( 1 - \frac{\ell}{N} \right)^{- 1 / 2} r_{h' (\rho), N} (y) -
   \left( \left( 1 + \frac{\ell}{2 N} \right) + c_1 (u^{\ell} - \ell \rho) +
   c_2 (u^{\ell} - \ell \rho)^2 \right) r_{h' (\rho), N} (0) \right| \\
   &\quad \leqslant \left| \left( \left( 1 - \frac{\ell}{N} \right)^{- 1 / 2} - 1 -
   \frac{1}{2} \frac{\ell}{N} \right) r_{h' (\rho), N} (y) \right| \\
   &\qquad + \left| \left( 1 + \frac{\ell}{2 N} \right) r_{h' (\rho), N} (y) -
   \left( \left( 1 + \frac{\ell}{2 N} \right) + c_1 (u^{\ell} - \ell \rho) +
   c_2 (u^{\ell} - \ell \rho)^2 \right) r_{h' (\rho), N} (0) \right|.
  \end{align*}
  The first term on the right hand side is $(g (\ell / N) - g (0) - g' (0)
  \ell / N) r_{h' (\rho), N} (y)$ for $g (u) = (1 - u)^{- 1 / 2}$, so since
  $\ell \leqslant N / 2$ and $r_{h' (\rho), N}$ is uniformly bounded in $y$
  and $\rho$ (recall Lemma \ref{lem:UB}), we get
  \[ \left| \left( \left( 1 - \frac{\ell}{N} \right)^{- 1 / 2} - 1 -
     \frac{1}{2} \frac{\ell}{N} \right) r_{h' (\rho), N} (y) \right| \lesssim
     \left( \frac{\ell}{N} \right)^2 . \]
  For the remaining term, we first choose
  \[ \tilde{c}_1 = \left( 1 + \frac{\ell}{2 N} \right) \frac{r_{h' (\rho), N}
     (0)^{- 1}}{\sigma_{h' (\rho)} \sqrt{N - \ell}} \frac{N^{- 1 / 2}}{2
     \sqrt{2 \pi}} \frac{m_{3, h' (\rho)}}{\sigma^3_{h' (\rho)}}, \hspace{2em}
     \tilde{c}_2 = - \left( 1 + \frac{\ell}{2 N} \right) \frac{r_{h' (\rho),
     N} (0)^{- 1}}{\sigma_{h' (\rho)}^2 (N - \ell)} \frac{1}{2 \sqrt{2 \pi}},
  \]
  for which we get from Lemma~\ref{lem:Gaussian expansion}
  \begin{align*}
   &\left| \left( 1 + \frac{\ell}{2 N} \right) r_{h' (\rho), N} (y) - \left(
   \left( 1 + \frac{\ell}{2 N} \right) + \tilde{c}_1 (u^{\ell} - \ell \rho)
   + \tilde{c}_2 (u^{\ell} - \ell \rho)^2 \right) r_{h' (\rho), N} (0)
   \right| \\
   &\qquad= \left( 1 + \frac{\ell}{2 N} \right) \left| r_{h' (\rho), N} (y) - r_{h'
   (\rho), N} (0) + \frac{N^{- 1 / 2}}{2 \sqrt{2 \pi}} \frac{m_{3, h'
   (\rho)}}{\sigma^3_{h' (\rho)}} y + \frac{1}{2 \sqrt{2 \pi}} y^2 \right| \\
   &\qquad\lesssim | y |^3 + N^{- 1 / 2} | y |^2 + N^{- 1} | y | \\
   &\qquad\lesssim N^{- 3 / 2} \left( 1 + \frac{| \ell \rho - u^{\ell} |^3}{\sigma_{h' (\rho)}^3} \right),
  \end{align*}
  where in the last step we used that $y = (\ell \rho - u^{\ell}) / \left(
  \sigma_{h' (\rho)} \sqrt{N - \ell} \right)$. It remains to control the
  difference between $c_1, \tilde{c}_1$ and $c_2, \tilde{c}_2$. First note
  that
  \[ r_{h' (\rho), N} (0) = r^0 (0) + N^{- 1} r_{h' (\rho)}^2 (0), \]
  so that we make an error of order $N^{- 2}$ when replacing $r_{h' (\rho), N}
  (0)$ by $r^0 (0) = (2 \pi)^{- 1 / 2}$ in $\tilde{c}_1$ and $\tilde{c}_2$.
  Similarly we can replace the factor $(1 - \ell / 2 N)$ by $1$ while making
  an error of order $\ell / N^2$. Finally we can replace $(N - \ell)^{- 1 /
  2}$ and $(N - \ell)^{- 1}$ by $N^{- 1 / 2}$ and $N^{- 1}$ respectively,
  while making an error of order $\ell / N^2$, which proves that
  \[ | c_1 - \tilde{c}_1 | + | c_2 - \tilde{c}_2 | \lesssim \frac{\ell}{N^2},
  \]
  and this concludes the proof.
\end{proof}

\begin{corollary}
  \label{cor:equivalence of ensembles}In the setting of Lemma~\ref{lem:equiv
  ensem} we have
  \[ \left| \psi_F (N, \rho) - \varphi_F (\rho) + \frac{\sigma_{h'
     (\rho)}^2}{2 N} \partial_{\rho \rho} \varphi_F (\rho) \right| \lesssim
     \left( \frac{\ell}{N} \right)^{3 / 2} (\tmop{var}_{h' (\rho)} (F))^{1 /
     2} . \]
\end{corollary}

\begin{proof}
  Start by noting that
  \[ \partial_{\lambda} (p_{\lambda} (u_0) \ldots p_{\lambda} (u_{\ell - 1}))
     = (u^{\ell} - \ell \rho' (\lambda)) p_{\lambda} (u_0) \ldots p_{\lambda}
     (u_{\ell - 1}), \]
  and therefore
  \[ \partial^2_{\lambda \lambda} (p_{\lambda} (u_0) \ldots p_{\lambda}
     (u_{\ell - 1})) = ((u^{\ell} - \ell \rho' (\lambda))^2 - \ell \rho''
     (\lambda)) p_{\lambda} (u_0) \ldots p_{\lambda} (u_{\ell - 1}) . \]
  So using that $\varphi_F (\rho) = E_{h' (\rho)} [F]$, we get
  \begin{align} \label{eq:phiF derivative} \nonumber
     \partial_{\rho} \varphi_F (\rho) & = E_{h' (\rho)} [(u^{\ell} - \ell \rho) F] h'' (\rho), \\
 \partial^2_{\rho \rho} \varphi_F (\rho) & = E_{h' (\rho)} [(u^{\ell} - \ell \rho)^2 F] (h'' (\rho))^2 - \ell E_{h' (\rho)} [F] h'' (\rho) + E_{h' (\rho)} [(u^{\ell} - \ell \rho) F] h''' (\rho) .
  \end{align}
  Since $h'$ is the inverse function of $\rho'$, the derivatives of $h'$ are
  given by
  \[ h'' (u) = \frac{1}{\rho'' (h' (u))}, \hspace{2em} h''' (u) = -
     \frac{1}{\rho'' (h' (u))^2} \rho''' (h' (u)) h'' (u) = - \frac{\rho'''
     (h' (u))}{\rho'' (h' (u))^3} = - \frac{m_{3, h' (u)}}{\sigma_{h' (u)}^6},
  \]
  where in the last step we used that
  \[ \rho (\lambda) = \log \int e^{\lambda u - V (u)} \mathd u =
     \log E_{\lambda_0} [e^{(\lambda - \lambda_0) u_0}] + \rho (\lambda_0)
  \]
  for any $\lambda_0$ and thus
  \[ \rho^{(m)} (\lambda) = \partial^m_{\gamma} E_{\lambda_0} [e^{\gamma u_0}]
     |_{\gamma = 0} = \kappa_{m, \lambda}, \]
  where $\kappa_m (\lambda)$ is the $m$-th cumulant, and that $\kappa_{3,
  \lambda} = m_{3, \lambda}$. Plugging this into~(\ref{eq:phiF derivative})
  (and using $\rho'' (\lambda) = \sigma_{\lambda}^2$) we get
  \[ \partial_{\rho} \varphi_F (\rho) = \frac{E_{h' (\rho)} [(u^{\ell} - \ell
     \rho) F]}{\sigma_{h' (\rho)}^2}, \]
  \[ \partial_{\rho \rho} \varphi_F (\rho) = \frac{1}{\sigma_{h' (\rho)}^2}
     \left( \frac{E_{h' (\rho)} [(u^{\ell} - \ell \rho)^2 F]}{\sigma^2_{h'
     (\rho)}} - \ell E_{h' (\rho)} [F] - \frac{m_{3, h' (u)}}{\sigma_{h'
     (u)}^4} E_{h' (\rho)} [(u^{\ell} - \ell \rho) F] \right), \]
  which shows that
   $\frac{\sigma_{h' (\rho)}^2}{2 N} \partial_{\rho \rho} \varphi_F (\rho) - \varphi_F (\rho)$ is equal to
  \begin{align*}
   - \left( 1 + \frac{\ell}{2 N} \right) \varphi_F (\rho) - \frac{1}{2 N}
   \frac{m_{3, h' (\rho)}}{\sigma^4_{h' (\rho)}} E_{h' (\rho)} [F (u^{\ell}
   - \ell \rho)] + \frac{1}{2 N \sigma_{h' (\rho)}^2} E_{h' (\rho)} [F
   (u^{\ell} - \ell \rho)^2],
  \end{align*}
  so that the claim follows from Lemma~\ref{lem:equiv ensem}.
\end{proof}


Finally we are ready to prove Proposition~\ref{prop:equiv-ensem}.

\begin{proof}[Proof of Proposition~\ref{prop:equiv-ensem}]
  Corollary~\ref{cor:equivalence of ensembles} gives
  \[ \left| \psi_F (N, \bar{u}^N) - \varphi_F (\bar{u}^N) + \frac{\sigma_{h'
     (\bar{u}^N)}^2}{2 N} (\partial_{\rho \rho} \varphi_F) (\bar{u}^N)
     \right|^2 \lesssim \left( \frac{\ell}{N} \right)^3 \sup_{\lambda}
     \tmop{var}_{\lambda} (F) . \]
  On the other hand, we obtain from a Taylor expansion
  \begin{align*}
   \varphi_F (\bar{u}^N)
   &= \varphi_F (\rho) + \partial_{\rho} \varphi_F (\rho) (\bar{u}_N - \rho) + \frac{1}{2} \partial_{\rho \rho} \varphi_F (\rho) (\bar{u}^N - \rho)^2 \\
   &\quad + \frac{1}{2} \int_0^1 (1 - \tau)^2 \partial_{\rho \rho \rho} \varphi_F (\rho + \tau (\bar{u}^N - \rho)) \mathd \tau (\bar{u}^N - \rho)^3.
  \end{align*}
  It is easy to see inductively that the $k$-th derivative of $\varphi_F$ is
  given by linear combinations of $E_{h' (\rho)} [(u^{\ell} - \ell \rho)^j F]$
  for $j \leqslant k$ multiplied with polynomials in $h^{(j)}$ for $1
  \leqslant j \leqslant k + 1$. Arguing as in Corollary~\ref{cor:equivalence
  of ensembles}, we see that all derivatives of $h$ are bounded in $\rho$.
  Moreover,
  \[ | E_{h' (\rho)} [(u^{\ell} - \ell \rho)^j F] | \leqslant E_{h' (\rho)}
     [(u^{\ell} - \ell \rho)^{2 j}]^{1 / 2} \tmop{var}_{h' (\rho)} (F)^{1 /
     2}, \]
  and
  \[ E_{h' (\rho)} [(u^{\ell} - \ell \rho)^{2 j}] = \sum_{i_1, \ldots, i_{2 j}
     = 0}^{\ell - 1} E_{h' (\rho)} [(u_{i_1} - \rho) \ldots (u_{i_{2 j}} -
     \rho)] \lesssim_j \ell^j m_{2 j, h' (\rho)}, \]
  where in the last step we used that $u_i - \rho$ and $u_{i'} - \rho$ are
  centered and independent, so that the expectation vanishes unless all
  variables appear at least in pairs, and therefore up to a combinatorial factor
  depending on $j$ the number of addends in the sum is of order $\ell^j$. In
  conclusion,
  \begin{align*}
   &E_{h' (\rho)} \left[ \left| \varphi_F (\bar{u}^N) - \varphi_F (\rho) -
   \partial_{\rho} \varphi_F (\rho) (\bar{u}_N - \rho) - \frac{1}{2}
   \partial_{\rho \rho} \varphi_F (\rho) (\bar{u}^N - \rho)^2 \right|^2
   \right] \\
   &\qquad= E_{h' (\rho)} \left[ \left| \frac{1}{2} \int_0^1 (1 - \tau)^2
   \partial_{\rho \rho \rho} \varphi_F (\rho + \tau (\bar{u}^N - \rho))
   \mathd \tau (\bar{u}^N - \rho)^3 \right|^2 \right] \\
   &\qquad \lesssim \ell^3 \sup_{\lambda} \tmop{var}_{\lambda} (F) E_{h' (\rho)}
   [(\bar{u}^N - \rho)^6] \lesssim \left( \frac{\ell}{N} \right)^3
   \sup_{\lambda} \tmop{var}_{\lambda} (F).
  \end{align*}
  Finally, we would like to replace $\frac{\sigma_{h' (\bar{u}^N)}^2}{2 N}
  (\partial_{\rho \rho} \varphi_F) (\bar{u}^N)$ by $\frac{\sigma^2_{h'
  (\rho)}}{2 N} \partial_{\rho \rho} \varphi_F (\rho)$. For the second factor
  we use the bound on $\partial_{\rho \rho \rho} \varphi_F$ that we just
  derived. The first factor can be estimated as follows:
  \[ | \partial_{\rho} \sigma^2_{h' (\rho)} | = | \partial_{\rho} (\rho'' (h'
     (\rho))) | = \left| \frac{\rho''' (h' (\rho))}{\rho'' (h' (\rho))}
     \right| \leqslant \sup_{\lambda} \left| \frac{\rho'''
     (\lambda)}{\sigma_{\lambda}^3} \right| \sigma_{\lambda} = \sup_{\lambda}
     \left| \frac{m_{3, \lambda}}{\sigma_{\lambda}^3} \right| \sigma_{\lambda}
     < \infty \]
     by Lemma \ref{lem:UB}.
\end{proof}

\end{document}